\documentclass[journal,10pt]{IEEEtran}
\IEEEoverridecommandlockouts

\usepackage[cmex10]{amsmath}
\usepackage{amsfonts,amssymb}
\usepackage{amsthm}
\usepackage{tikz}
\usepackage{pgfplots}
\usepackage{bm}
\usepackage{array}
\usepackage[hidelinks]{hyperref}
\usepackage[normalem]{ulem}
\usepackage{cite}
\usepackage{xr}

\ifCLASSOPTIONcompsoc
 \usepackage[caption=false,font=normalsize,labelfont=sf,textfont=sf]{subfig}
\else
 \usepackage[caption=false,font=footnotesize]{subfig}
\fi

\theoremstyle{plain}
\newtheorem{thm}{Theorem}
\newtheorem{lem}[thm]{Lemma}
\newtheorem{col}[thm]{Corollary}
\newtheorem{prop}[thm]{Proposition}

\newtheorem{assum}[thm]{Assumption}
\newtheorem{dfn}[thm]{Definition}

\newtheorem{rem}{Remark}

\usepackage{mydefs}
\usepackage{color,soul}
\usepackage{tabu,multirow}
\newcommand{\YT}[1]{{\color{blue} #1}}

\def\jm#1{\textcolor{black}{#1}}

\allowdisplaybreaks

\begin{document}

\title{Distributed Algorithms for Composite Optimization: Unified Framework and Convergence Analysis}

 \author{Jinming Xu$^{\dagger}$\thanks{$^\dagger$College of Control Science and Engineering, Zhejiang University, Hangzhou 310027, China. Email:~\texttt{jimmyxu@zju.edu.cn}.}, Ye Tian$^\ddagger$, Ying Sun$^\ddagger$, and Gesualdo Scutari$^\ddagger$\thanks{$^\ddagger$School of Industrial Engineering, Purdue University, West-Lafayette, IN, USA. Emails: \texttt{<tian110, sun578, gscutari>} \texttt{@purdue.edu.} This work has been supported by the  USA NSF
Grants CIF 1632599 and CIF 1564044; and the ARO  Grant
W911NF1810238.

Part of this work has been presented at IEEE CAMSAP 2019 \cite{ XuSunScutariJ,XUCAMSAP_arxiv}.
}\vspace{-0.6cm}}

\maketitle


\begin{abstract}
We study distributed composite  optimization over networks: agents minimize  a sum of smooth (strongly) convex  functions--the agents' sum-utility--plus a nonsmooth (extended-valued) convex one. We propose a  general {\it unified}  algorithmic framework for such a class of problems and provide a   unified convergence analysis leveraging   the theory of operator splitting. 
Distinguishing features of our scheme are: 
(i)  When the agents' functions are strongly convex, the algorithm converges at a {\it linear} rate,  whose dependence on the agents' functions   and   network topology is {\it decoupled}, matching the typical rates of centralized optimization; the rate expression   improves on existing results;  (ii) When the objective function is convex (but not strongly convex), similar separation as in (i) is established for the  coefficient of the proved sublinear rate; 
 (iii) The algorithm can adjust the ratio between the number of communications and computations to achieve a rate (in terms of computations) independent on the network connectivity; 
 and (iv)  
A by-product of our analysis is a tuning recommendation for several existing (non accelerated) distributed algorithms yielding the fastest provably (worst-case) convergence rate. 
This is the first time that a general distributed algorithmic framework  applicable to {\it composite} optimization enjoys all such properties.
\vspace{-0.4cm}
\end{abstract}

\IEEEpeerreviewmaketitle

\section{Introduction}

We study distributed multi-agent  optimization over networks, modeled as undirected static graphs. Agents aim at solving \vspace{-0.2cm}
\begin{equation}\label{prob:dop_nonsmooth_same}
\min_{x\in\mathbb{R}^{d}} F(x)+G(x),\quad F(x)\triangleq {\frac{1}{m}}\sum_{i=1}^m f_i(x),\tag{P}
\end{equation}
where $f_i:\mathbb{R}^d\to \mathbb{R}$ is the cost function of agent $i$, assumed to be $L$-smooth,  $\mu$-strongly convex (with $\mu\geq 0$), and known only to the agent; and $G:\mathbb{R}^d\to \mathbb{R}\cup\{-\infty, \infty\}$ is a  nonsmooth, convex (extended-value) function, which can be used to enforce shared constraints  or specific structures on the solution (e.g., sparsity).

The focus of this paper is the design of a {\it unified} (first-order) algorithmic framework for  Problem \eqref{prob:dop_nonsmooth_same} with provably  convergence rate. 
When $G=0$ and $\mu>0$,  several distributed schemes have been proposed in the literature that enjoy {\it linear} rate; examples include EXTRA~\cite{shi2015extra}, AugDGM~\cite{xu2015augmented}, NEXT~\cite{di2016next}, Harnessing~\cite{qu2017harnessing}, SONATA \cite{YingMAPR,sun2019convergence}, DIGing~\cite{nedich2016achieving},   NIDS~\cite{li2017decentralized},   Exact Diffusion~\cite{yuan2018exact_p1},  MSDA~\cite{scaman17optimal}, and the distributed algorithms in \cite{jakovetic2018unification,mansoori2019general}. When  $\mu=0$  and still $G=0$, a sublinear rate of $O(1/k)$ ($k$ counts the number of gradient evaluations) is   achieved by some of the above methods \cite{di2016next,qu2017harnessing,nedich2016achieving,xu2015augmented}  and   other primal-dual schemes, including  D-ADMM~\cite{wei2012distributed_admm}. Results for  $G\neq 0$ are relatively scarce; to our knowledge, the only two schemes achieving linear rate for \eqref{prob:dop_nonsmooth_same} are SONATA \cite{sun2019convergence} and the one in \cite{alghunaim2019linearly}; the former under the assumption that $F$ is strongly convex and the latter requiring each $f_i$ to be so. Sublinear rate of $O(1/k)$ has been proved for a variety of schemes, including  PG-EXTRA~\cite{shi2015proximal}, D-FBBS~\cite{xu2018bregman} and DPGA~\cite{aybat2017distributed}.  

{Although the aforementioned algorithms all achieve linear  or   sublinear convergence rates, they differ in the nature and strength of their convergence guarantees.}   No {\it unified} algorithmic design and convergence analysis can be inferred by existing studies. 
Furthermore, for most of the schemes, one notices  a gap between theory and practice:  convergence analyses yield tuning recommendations and associated rate bounds  that numerical simulations prove being far too conservative. To make these algorithms work in practice, practitioners  often use manual,  ad-hoc tunings. This  however makes the comparison  of  different schemes hard, running the risk of drawing  misleading conclusions. 
These issues  suggest the following questions: 
\begin{itemize}
\item[{\bf (Q1)}] Can one unify the design and analysis of distributed algorithms for Problem \eqref{prob:dop_nonsmooth_same}? 
\item[{\bf (Q2)}]  How do provable  rates of such schemes compare each other and with that  of the centralized proximal-gradient algorithm applied to   \eqref{prob:dop_nonsmooth_same}?
\end{itemize}
\textbf{On (Q1):} Recent efforts toward a better understanding of the taxonomy of  distributed algorithms   are the following: \cite{jakovetic2018unification} provides a connection between EXTRA and DIGing; \cite{Scoy-canonical18} provides a canonical representation of some of the distributed algorithms above--NIDS and Exact-Diffusion are proved to be equivalent; and \cite{Sundararajan17} provides an automatic (numerical) procedure to prove linear rate of some classes of distributed algorithms.  These efforts model only first-order distributed algorithms applicable to Problem \eqref{prob:dop_nonsmooth_same} {\it with $G=0$} and employing a {\it single} round of communication and gradient computation. 
However, existing algorithms have their rate analysis done in isolation, under ad-hoc convergence conditions and different ranges for the stepsize--see Table~\ref{tab:optim_linear_rate}.  
{For instance,   NIDS~\cite{li2017decentralized} and    Exact Diffusion~\cite{yuan2018exact_p1}  are proved to be equivalent \cite{jakovetic2018unification,Scoy-canonical18};}
this   however is not reflected   by the convergence analyses and associated   rate bounds and admissible stepsize values, which instead are quite different.


\noindent \textbf{On (Q2):} Question (Q2) has been only partially addressed in the literature.  For instance, MSDA~\cite{scaman17optimal} uses multiple communication steps to achieve the lower complexity bound of  \eqref{prob:dop_nonsmooth_same} when $\mu>0$ and $G=0$; the OPTRA algorithm \cite{xu2019accelerated} achieves the lower bound when  $\mu=0$ (still and $G=0$); and
the algorithms in \cite{van2019distributed} and \cite{li2017decentralized} achieve linear rate and can adjust the number of communications performed at each iteration to  match the rate of the centralized gradient descent. 
However it is not clear how to extend  (if possible) these methods and their convergence analysis to the more general composite ($G\neq 0$) setting \eqref{prob:dop_nonsmooth_same}.
Furthermore, even when $G= 0$,   the rate results of existing algorithms are not theoretically comparable with each other--see Table~\ref{tab:optim_linear_rate}; they have been obtained under different stepsize range values and technical assumptions (e.g., on the weight matrices). 
   {Similarly, when $\mu=0$,   EXTRA~\cite{shi2015extra}, DIGing~\cite{qu2017harnessing,nedich2016achieving} D-ADMM~\cite{wei2012distributed_admm}, and  PG-EXTRA~\cite{shi2015proximal}, D-FBBS~\cite{xu2018bregman}, DPGA~\cite{aybat2017distributed}   achieve a sublinear rate of $O(1/k)$ for $G=0$ and $G\neq 0$, respectively.   However, the rate expression given in terms of ``big-O'' notation lacks of any insight on the dependence of the rate on the key design parameters (e.g., the stepsize).}


\begin{table*}[ht]
\caption{Convergence Properties of Distributed Algorithms for $L$-Smooth and $\mu$-Strongly Convex  $f_i$ ($\mu>0$): Existing results vs. improvements from this paper.  
}
\centering
\label{tab:optim_linear_rate}
\renewcommand{\arraystretch}{1.8}
\begin{tabular}{c|c|c|c|c|c}
\hline
\multicolumn{6}{c}{$L$-Smooth and $\mu$-Strongly Convex Functions [$\rho=\lambda_{\max}(W-J)]$, $W\in \mathbb{S}^m$, and $-I \prec W \preceq I$]} \\
\hline
\multirow{2}{*}{{\bf Algorithm}} &  \multirow{2}{*}{{\bf Problem} }  &   \multicolumn{2}{c|}{{\bf Stepsize }}  & \multicolumn{2}{c}{ {\bf Rate: $\mathcal{O}\left( \delta \log(\frac{1}{\epsilon})\right)$}} \\
\cline{3-6}
 &   &  {\bf literature (upper bound) } &  {\bf this paper (optimal, Corollary \ref{col:optimal_D})} &$\delta$,   {\bf literature} &   $\delta$,  {\bf  this paper} \\
\hline \hline
EXTRA~\cite{shi2015extra} & $F$  &$ \mathcal{O}\left(\frac{\mu (1-\rho)}{L^2} \right)$ & $\frac{2}{2L/(1-\rho)+\mu}=\mathcal{O}\left( \frac{1-\rho}{L}\right) $ & $\frac{L^2\kappa^2}{1-\rho}$ &    $\frac{\kappa}{1-\rho}$\\ 
\hline
NEXT~\cite{di2016next}  & \multirow{2}{*}{$F$}    & \multirow{2}{*}{$\mathcal{O}\left( \frac{(1-\rho)^2}{L} \right)$}    &  \multirow{2}{*}{$\frac{2}{L+\mu}=\mathcal{O}\left( \frac{1}{L} \right)$} & \multirow{2}{*}{N.A. }          &   \multirow{2}{*}{$\max\left\{\kappa, ~\frac{1}{(1-\rho)^2} \right\}$ }\\
AugDGM~\cite{xu2015augmented,xu2017convergence} &  &  &  &    &   \\
\hline
DIGing~\cite{nedich2016achieving,nedic2017geometrically}   & $F$     &$ \mathcal{O}\left( \min\{\frac{1}{\sqrt{n}\sqrt{\kappa} L \rho (1-\rho)},\frac{1}{L}\}\right)$ & $\frac{2}{4L/(1-\rho)^2+\mu} = \mathcal{O}\left( \frac{(1-\rho)^2}{L} \right)$    &$\max\{\kappa,\, \gg \frac{1}{1-\rho}\}$  &$\frac{\kappa}{(1-\rho)^2}$  \\
\hline
Harnessing~\cite{qu2017harnessing}   & $F$       &$ \mathcal{O}\left(\frac{(1-\rho)^2}{\kappa L}  \right)$  &  $\frac{2}{4L/(1-\rho)^2+\mu} = \mathcal{O}\left( \frac{(1-\rho)^2}{L} \right)$      &$\frac{\kappa^2}{(1-\rho)^2}$ & $\frac{\kappa}{(1-\rho)^2}$ \\
\hline
NIDS~\cite{li2017decentralized}    & $F$      &$ \mathcal{O}\left(\frac{1}{L} \right) $  &    $\frac{2}{L+\mu} = \mathcal{O}\left( \frac{1}{L} \right) $   &$\max\{\kappa,\, \frac{1}{1-\rho}\}$      &    $\max\{\kappa,\, \frac{1}{1-\rho}\}$   \\
\hline
Exact Diffusion~\cite{yuan2018exact_p1}    & $F$     &$ \mathcal{O}\left( \frac{\mu}{L^2}\right) $ & $\frac{2}{L+\mu}=  \mathcal{O}\left(\frac{1}{L} \right) $ &   $\gg \max \{ \kappa^2, \, \frac{1}{1-\rho}\}$   & $\max\{\kappa,\, \frac{1}{1-\rho}\}$     \\
\hline
\cite{jakovetic2018unification} $(b=0)$    & $F$      & $\mathcal{O}\left( \frac{(1-\rho)^2}{\kappa L} \right) $    &   $\frac{2}{L/\lambda_{\min}(W^2)+\mu} = \mathcal{O}\left( \frac{\lambda_{\min}(W^2)}{L} \right)$       & $ \frac{\kappa^2}{(1-\rho)^2}$        &  $\frac{\kappa}{1-\rho}$   \\
\hline
\cite{jakovetic2018unification} $(b=\frac{1}{\gamma}W, \, W \succ 0)$     & $F$      & N.A.    &   $\frac{2}{L/\lambda_{\min}(W)+\mu} =  \mathcal{O}\left( \frac{\lambda_{\min}(W)}{L} \right)$    & N.A.        &  $\frac{\kappa}{1-\rho}$   \\
\hline
\multirow{2}{*}{\cite{mansoori2019general} $(W \succ 0)$}      & \multirow{2}{*}{$F$}     &  \multirow{2}{*}{(14) in the paper} &         $\frac{2}{\mu + L/((1-\lambda_{\min}(W))\lambda_{\min}(W)^K)} $          &  \multirow{2}{*}{N.A. }                     &   $\max \big\{ \frac{1}{1-\rho^K}, $  \\
                      &         &                 &          $=\mathcal{O}\left( \frac{(1-\lambda_{\min}(W))\lambda_{\min}(W)^K}{L} \right)$             &                    &      $\frac{\kappa}{(1-\lambda_{\min}(W))\lambda_{\min}(W)^K} \big\}$  \\
 \hline
\cite{alghunaim2019linearly}      & $F+G$       & $\frac{2\lambda_{\min}(W)}{L+\mu} = \mathcal{O}\left( \frac{\lambda_{\min}(W)}{L}\right) $  &     $\frac{2\lambda_{\min}(W}{L+\mu\lambda_{\min}(W)} = \mathcal{O}\left( \frac{\lambda_{\min}(W)}{L}\right) $           & $> \frac{\kappa}{1-\rho}$   &  $\max\{\frac{\kappa}{\lambda_{\min}(W)}, \frac{1}{\alpha (1-\rho)}\}$   \\
\hline
this paper    & $F+G$     & \multicolumn{2}{c|}{ $\frac{2}{L+\mu}  $} & \multicolumn{2}{c}{$\max\left\{\kappa, ~\frac{1}{1-\rho} \right\}$} \\
\hline
\end{tabular}
\end{table*}

\indent This paper aims at addressing Q1 and Q2 
in the general setting \eqref{prob:dop_nonsmooth_same}, with either $\mu>0$ or $\mu=0$. Our major contributions are discussed next.
\noindent \textbf{1) Unified framework and rate analysis:} We propose a general primal-dual  distributed  algorithmic framework   that unifies for the first time   ATC (Adapt-Then-Combine)-  {\it and} CTA (Combine-Then-Adapt)-based distributed algorithms, solving either smooth ($G=0$) or {\it composite} optimization problems ($G\neq0$). Most of existing ATC and CTA schemes are special cases of the proposed framework--cf. Table~\ref{tab:instances}. A unified set of convergence conditions and rate expression are provided, leveraging   a novel operator contraction-based analysis. By product of our unified framework and convergence conditions, several existing schemes, proposed  only to solve smooth instances of \eqref{prob:dop_nonsmooth_same} \cite{shi2015extra,li2017decentralized,yuan2018exact_p1,qu2017harnessing,di2016next,xu2015augmented}, gain now their ``proximal'' extension and thus become applicable also to composite optimization while  enjoying the same (novel) convergence rate (as derived in this paper)  of their ``non-proximal'' counterparts. \textbf{2) Improving upon existing results and tuning recommendations:}  Our results improve on existing convergence conditions and rate bounds,   such as \cite{shi2015extra,li2017decentralized,yuan2018exact_p1,qu2017harnessing,di2016next,xu2015augmented}--Table~\ref{tab:optim_linear_rate} shows the improvement achieved by our analysis in terms of stepsize bounds and rate expression (see  Sec. \ref{sec:discussion} for more details). 
Our rate results provide for the first time a platform for a fair   comparison of these algorithms; the tightness of our rates as well as the established ranking of the algorithms based on the new rate expressions are supported  by  numerical results. 
\textbf{3) Rate separation when $G\neq 0$:} For ATC-based schemes, when $\mu>0$, the dependency of the linear rate  on the agents' functions   and the network topology are {\it decoupled}, matching the typical rates of the proximal gradient algorithm applied to \eqref{prob:dop_nonsmooth_same} and  consensus averaging. Furthermore, 
the optimal stepsize value 
is independent on the network and  matches the optimal choice for the centralized proximal gradient algorithm. 
When $\mu=0$, we provide an explicit expression of the sublinear rate (beyond the ``Big-O'' decay) revealing a similar decoupling between optimization and network parameters. Our novel expression sheds also light on the choice of the stepsize minimizing the rate bound: the optimal choice  is not necessarily $1/L$ but  instead   depends on the network parameters as well as the degree of heterogeneity of the agents' functions (cf. Sec. \ref{sec:sublinear}). These results are a major departure from existing analyses,  which do not show such a clear separation, and complements the results in \cite{li2017decentralized} applicable only to smooth and strongly convex instances of  \eqref{prob:dop_nonsmooth_same}. 
\textbf{4) Balancing computation and communication:} When $\mu>0$, the proposed scheme can  adjust the ratio between the number of communication and computation steps to achieve the {\it same} rate of the centralized proximal gradient scheme. We show that Chebyshev acceleration can also be employed to further reduce the number of communication steps per computation.

The results of this work have been partially presented in \cite{XuSunScutariJ}.  While  preparing the final version of this manuscript, we noticed the arxiv submission  \cite{alghunaim2019decentralized}, which is an independent and parallel work (cf. \cite{XUCAMSAP_arxiv}). There are  some substantial differences between our findings and \cite{alghunaim2019decentralized}: i) our algorithmic framework  unifies ATC and CTA schemes while \cite{alghunaim2019decentralized} can cover only ATC ones;    our analysis is based on an operator contraction-based analysis, which is of independent interest; and  ii) we  study  convergence also when $F$ is convex but not strongly convex while  \cite{alghunaim2019decentralized} focuses only on strongly convex problems.

 \vspace{-0.2cm}

\section{Problem Statement}
We study Problem~\eqref{prob:dop_nonsmooth_same} under   the following  assumption, capturing either strongly convex or just convex objectives.
\begin{assum}\label{assum:smooth_strong_reg}
Each   $f_i:\mathbb{R}^d\rightarrow\mathbb{R}$ is  $\mu$-strongly convex, $\mu\geq 0$,  and $L$-smooth; and $G:\mathbb{R}^d\rightarrow\mathbb{R}\cup\{\pm\infty\}$ is proper, closed and convex. When $\mu>0$, define $\kappa\triangleq L/\mu$.
\end{assum} 


\smallskip

\noindent {\bf Network model:} Agents are embedded in a network, modeled as  an undirected, static graph $\Gh=(\Vx,\Eg)$, where $\Vx$ is the set of nodes (agents) and $\{i,j\}\in\Eg$ if there is an edge (communication link) between node $i$ and $j$. 
We make the blanket assumption that $\Gh$ is connected.  We introduce the following matrices associated with $\mathcal{G}$, which will be used to build the proposed distributed algorithms. \vspace{-0.1cm}
\begin{dfn}[Gossip matrix]
\label{dfn:weight_matrix}
A matrix $\W\triangleq [W_{ij}]\in \mathbb{R}^{m\times m}$ is said to be compliant to the graph $\Gh=(\Vx,\Eg)$ if  $W_{ij} \neq 0$ for  $\{i,j\}\in\Eg$, and $W_{ij}=0$ otherwise. The set of such matrices is denoted by $\mathcal{W}_\mathcal{G}$.
\end{dfn}
\begin{dfn}[$K$-hop gossip matrix]
\label{dfn:k-hop_weight_matrix}
Given $K\in\mathbb{N}_+,$ a matrix $ {W}'\in\mathbb{R}^{m\times m}$ is said to be a $K$-hop gossip matrix associated to $\Gh=(\Vx,\Eg)$  if $ {W}'=P_K(\W)$, for some $\W\in\mathcal{W}_\Gh$, where   $P_K(\cdot)$ is a monic polynomial of order $K$.\vspace{-0.1cm}
\end{dfn}
Note that, if  $\W\in \mathcal{W}_\Gh$,  using $W_{ij}$ to   linearly combine information between two immediate neighbor  agents $i$ and $j$ corresponds to performing a single communication  round.   Using  a $K$-hop matrix ${W}'=P_K(\W)$  requires instead $K$ consecutive rounds of communications. 
$K$-hop gossip matrices are crucial to employ acceleration of the communication step, which will be a key ingredient to exploit the tradeoff between  communications and computations (cf.~ Sec.~\ref{sec:tradeoff}).\\
\noindent{\textbf{A saddle-point reformulation:}} Our path to design distributed solution methods for \eqref{prob:dop_nonsmooth_same} is to solve a   saddle-point reformulation of  \eqref{prob:dop_nonsmooth_same}   via general proximal splitting algorithms that are implementable over  $\mathcal{G}$.   Following a standard path in the literature,
we introduce local copies $x_i\in \mathbb{R}^d$ (the $i$-th one is owned by agent $i$) of  $x$  and functions \vspace{-0.2cm}\begin{equation}
	f(\x)\triangleq\sum_{i=1}^mf_i(x_i)\quad \text{and}\quad g(\x)\triangleq\sum_{i=1}^m G(x_i),\vspace{-0.1cm}
\end{equation} with $\x\triangleq [x_1,\ldots, x_m]^\top\in\mathbb{R}^{m\times d}$;
  \eqref{prob:dop_nonsmooth_same} can be  rewritten as\vspace{-0.1cm}
\begin{equation}\label{prob:dop_nonsmooth_same_augmented}
\min_{\x\in\mathbb{R}^{m\times d}} f(\x)+g(\x),~{\text{s.t.}}~  \sqrt{{C}}\x=0,\vspace{-0.1cm}
\end{equation}
where $C$ satisfies the following assumption:
\begin{assum}\label{assum:cond_C}
$C \in \mathbb{S}^m_+$ and $\Null{{C}}=\Span{\ones}$.
\end{assum}
Under this condition,  the constraint $\sqrt{ {C}}\x=0$   enforces a consensus among $x_i$'s and thus  \eqref{prob:dop_nonsmooth_same_augmented} is equivalent to   \eqref{prob:dop_nonsmooth_same}. 
 {The set of points satisfying the KKT conditions of \eqref{prob:dop_nonsmooth_same_augmented} reads:}
\begin{align}\label{eq:kkt_conditions}
 \mathcal{S}_{\texttt{KKT}}  \triangleq & \left\{   \x \in  \mathbb{R}^{m\times d} \, \big\vert\, \exists\, \y \in  \mathbb{R}^{m\times d} \text{ such that} \right.\nonumber\\
& \left.  \sqrt{{C}}\x=0, \quad  \nabla f(X)+  \sqrt{{C}}\y \in-\partial g(\x)\right\},
\end{align}
where $\nabla f(\x)\triangleq [\nabla f_1(x_1),\nabla f_2(x_2),...,\nabla f_m(x_m)]^\top$
and  $\partial g(\x)$ denotes the subdifferential  of $g$ at $\x$. Then we have the following standard result. 

\begin{lem}\label{lemma_eq_KKT} Under Assumption \ref{assum:smooth_strong_reg}, $x^\star\in \mathbb{R}^d$ is an optimal solution of Problem~\eqref{prob:dop_nonsmooth_same} if and only if  ${1}_m x^{\star\top} \in \mathcal{S}_{\texttt{KKT}}$.  
\end{lem}

Building on Lemma \ref{lemma_eq_KKT}, 
in the next section, we propose a general distributed algorithm for \eqref{prob:dop_nonsmooth_same}  based on  a suitably defined operator splitting solving  the KKT system   \eqref{eq:kkt_conditions}. 
 \vspace{-0.2cm}

\section{A General Primal-Dual Proximal Algorithm}\label{sec:alg-des}

\begin{table*}[ht]
\caption{Special cases of Algorithm~\eqref{alg:g-ABC} for specific choices of   $A,B,C$ matrices and given  gossip matrix $-\I\prec \W \preceq \I$. The listed algorithms can be cast in the proposed framework. 
}
\centering
\label{tab:instances}
\renewcommand{\arraystretch}{1.8}
\begin{tabular}{c|c|c|c}
\hline
\multirow{2}{*}{{\bf Algorithm}} & \multirow{2}{*}{{\bf Problem}} & \multirow{2}{*}{{\bf Choice of the $A,B,C$}}    & \# {\bf communications} \\
&&&    {\bf per gradient evaluation} \\
\hline \hline
EXTRA~\cite{shi2015extra} &$F$ &$A=\frac{I+W}{2} \quad B=I  \quad C=\frac{I-W}{2}$ & 1 \\
\hline
NEXT~\cite{di2016next}/AugDGM~\cite{xu2015augmented,xu2017convergence}  &$F$  & $A=\left( \frac{I+W}{2} \right)^2 \quad B=\left( \frac{I+W}{2} \right)^2  \quad C=\left( \frac{I-W}{2} \right)^2$  & 2\\
\hline
DIGing~\cite{nedich2016achieving,nedic2017geometrically}/Harnessing~\cite{qu2017harnessing} &$F$  & $A=\left( \frac{I+W}{2} \right)^2  \quad B=I    \quad C=\left( \frac{I-W}{2} \right)^2$ & 2 \\
\hline
NIDS~\cite{li2017decentralized}/Exact Diffusion~\cite{yuan2018exact_p1}  &$F$       & $A=\frac{I+W}{2}  \quad B=\frac{I+W}{2}   \quad C=\frac{I-W}{2}$  & 1 \\
\hline
\cite{jakovetic2018unification} ($B'=bI$) &$F$  &$A=W^2 + \gamma b(I-W)   \quad B=\I \quad C=(I-W)^2 + \gamma b(I-W) $ &2 \\
\hline
\cite{mansoori2019general}   &$F$  & $A= \W^K      \quad B=\sum_{i=0}^{K-1}{\W^i}   \quad C=\I-\W^K $ & $K$ \\
 \hline
\cite{alghunaim2019linearly} &$F+G$ & $A=W   \quad B=I     \quad C=\alpha(I-W)$ with $0\prec W \preceq I$ and $\alpha \leq 1$ & 1\\
\hline
\end{tabular}
\end{table*}

The proposed general primal-dual proximal   algorithm, termed $ABC-$Algorithm, reads
 \vspace{-0.1cm}
\begin{subequations}\label{alg:g-ABC}
\begin{align}
\x^k&=\prox{\gamma g}{\z^k}, \label{alg:g-ABC_z}\\
\z^{k+1}&=\A\x^k-\gamma\B\nabla f(\x^k)-Y^k, \label{alg:g-ABC_x}\\
Y^{k+1}&=Y^k+{\C}\z^{k+1}, \label{alg:g-ABC_y}
\end{align}
\end{subequations}
with  $\z^0\in \mathbb{R}^{m\times d}$ and $\y^0=0$. In \eqref{alg:g-ABC_z},  $\prox{\gamma g}{\x}\triangleq\text{arg}\min_{\y}g(\y)+\frac{1}{2\gamma}\norm{\x-\y}^2$ is the standard proximal operator. Eq. \eqref{alg:g-ABC_z} and~\eqref{alg:g-ABC_x} represent the update of the primal variables, where  $\A,\B\in \mathbb{R}^{m\times m}$  are suitably chosen weight matrices, and $\gamma>0$ is the stepsize. Eq.~\eqref{alg:g-ABC_y} represents the update of  the dual variables.
\\\indent
Define the  set\vspace{-0.1cm}\begin{align}
	 \hspace{-0.4cm}\mathcal{S}_{\texttt{Fix}}  &\triangleq   \Big\{   \x \in  \mathbb{R}^{m\times d} \, \big\vert\, \C\x=0 \text{ and }\nonumber\\& \quad 
  \ones^\top(I-\A)\x+\gamma \, \ones^\top\B\nabla f(\x)\in - \gamma\,  \ones^\top\partial g(\x) \Big\}.
\end{align}
Since all agents share the same $G$, it is not difficult to check that any fixed point $(X^\star, Z^\star, Y^\star)$ of  Algorithm \eqref{alg:g-ABC} is such that $X^\star \in \mathcal{S}_{\texttt{Fix}}$.
The following are {\it necessary} and {\it sufficient} conditions on  $\A, \B$   for  $X^\star \in \mathcal{S}_{\texttt{Fix}}$   to be a solution of \eqref{prob:dop_nonsmooth_same_augmented}.   


\begin{assum}\label{assum:cond_A_B}
The weight matrices $\A,\B \in \mathbb{R}^{m\times m}$ satisfy:
$\ones^\top\A \,\ones=m$,  and  $\ones^\top\B = \ones^\top$.
\end{assum}
\begin{lem}\label{lem:iff_fixed_point}
Under Assumption~\ref{assum:cond_C},  $\mathcal{S}_\texttt{KKT}=\mathcal{S}_\texttt{Fix}$ if and only if $\A,\B$ satisfy Assumption~\ref{assum:cond_A_B}.
\end{lem}
\begin{proof}
See Sec. \ref{sec:pf_lm2} in Appendix.\vspace{-0.2cm}
\end{proof}

\subsection{Connections with existing distributed algorithms} \label{Sec:special-cases} 
Algorithm~\eqref{alg:g-ABC} contains a gamut of distributed (and centralized) schemes, corresponding to different choices of the weight matrices $\A,\B$ and $\C$; any   $A,B,C\in \mathcal{W}_{\mathcal{G}}$ leads to  distributed  implementations. The use of general matrices $\A$ and $\B$ (rather the more classical choices $\A=\B$ or $\B=I$) permits a unification of both ATC- and CTA-based  updates; this includes several existing distributed algorithms proposed for special cases of (P), as    discussed next. 
  
  We begin rewriting  Algorithm~\eqref{alg:g-ABC} in the following equivalent form by    subtracting   \eqref{alg:g-ABC_x} at iteration $k+1$  from \eqref{alg:g-ABC_x} at iteration $k$: 
\begin{align}\label{eq:g-pd-ATC_eliminate_y}
\z^{k+2}&=(\I-\C)\z^{k+1}+\A(\x^{k+1}-\x^k)\nonumber\\
&\quad -\gamma\B(\nabla f(\x^{k+1})-\nabla f(\x^k)),
\end{align}
where $\x^k=\prox{\gamma g}{\z^k}$.

When $G=0$, \eqref{eq:g-pd-ATC_eliminate_y} reduces to
\begin{equation}\label{eq:g-pd-ATC_eliminate_y_G=0}
\begin{aligned}
\x^{k+2}&=(\I-\C+\A)\x^{k+1}\\
&~~~-\A\x^k-\gamma\B(\nabla f(\x^{k+1})-\nabla f(\x^k)).
\end{aligned} 
\end{equation}

We show next  that   the schemes in \cite{shi2015extra,li2017decentralized,yuan2018exact_p1,di2016next,xu2015augmented,nedich2016achieving,qu2017harnessing,jakovetic2018unification,mansoori2019general,alghunaim2019linearly} are all special cases of Algorithm \eqref{alg:g-ABC}.  Table \ref{tab:instances} summarizes 
the specific choices of  $A, B$, and $C$ in   \eqref{alg:g-ABC} yielding  the desired equivalence, where $\W\in  \mathcal{W}_\mathcal{G}$ is the weight matrix used in the target distributed algorithms. Notice that all these choices  
satisfy Assumptions \ref{assum:cond_C} and \ref{assum:cond_A_B}. 

 

 
\noindent {\bf 1) EXTRA \cite{shi2015extra}:}    EXTRA solves \eqref{prob:dop_nonsmooth_same}  with $G=0$, and  reads
\begin{equation}\label{eq:EXTRA}
\x^{k+2}=(\I+\W)\x^{k+1}-\tilde{\W}\x^k-\gamma(\nabla f(\x^{k+1})-\nabla f(\x^k)),
\end{equation}
where $\W,\,\tilde{\W}$ are two design weight matrices satisfying  $({\I+\W})/2\succeq\tilde{\W}\succeq\W$ and $\tilde{\W}\succ 0$. 
Clearly,  \eqref{eq:EXTRA} is an instance of \eqref{eq:g-pd-ATC_eliminate_y_G=0} [and thus \eqref{alg:g-ABC}], with    $\A=\tilde{\W}$,   $\B=\I,$ and $\C=\tilde{\W}-\W$.  
\smallskip 

\noindent {\bf 2) NIDS~\cite{li2017decentralized}~/~Exact diffusion~\cite{yuan2018exact_p1,yuan2018exact_p2}:}
The NIDS (Exact Diffusion) algorithm applies to \eqref{prob:dop_nonsmooth_same}  with $G=0$, and reads
\begin{equation}
\x^{k+2}=\frac{\I+\W}{2}(2\x^{k+1}-\x^k-\gamma(\nabla f(\x^{k+1})-\nabla f(\x^k))),
\end{equation}
which is an instance of our general scheme, with $\A=\B=({\I+\W})/{2}$ and $\C=({\I-\W})/{2}$. 
\smallskip

\noindent {\bf 3) NEXT~\cite{di2016next} \& AugDGM~\cite{xu2015augmented}:}
The gradient tracking-based  algorithms  NEXT/AugDGM applied to \eqref{prob:dop_nonsmooth_same}  with $G=0$, are:
\begin{subequations}\label{alg:NEXT/AugDGM}
\begin{align}
\x^{k+1}&=\W(\x^k-\gamma Y^k), \label{alg:g-gradtrack-ATC_x}\\
Y^{k+1}&=\W(Y^k+\nabla f(\x^{k+1})-\nabla f(\x^k)).\label{alg:g-gradtrack-ATC_y}
\end{align}
\end{subequations}
Eliminating the $Y$-variable, \eqref{alg:NEXT/AugDGM} can be rewritten as: 
\begin{equation}\label{eq:g-gradtrack-ATC_eliminate_y}
\x^{k+2}=2\W\x^{k+1}-\W^2\x^k-\gamma\W^2(\nabla f(\x^{k+1})-\nabla f(\x^k)).
\end{equation}
Clearly (\ref{eq:g-gradtrack-ATC_eliminate_y}) is an instance of our scheme \eqref{alg:g-ABC}, with $\A=\B=\W^2,\C=(\I-\W)^2$.
Notice that distributed gradient tracking schemes in the so-called CTA form are also special cases of Algorithm \eqref{alg:g-ABC}. For instance, one can show that the DIGing algorithm~\cite{nedich2016achieving} corresponds to the setting $\A=\W^2,\B=\I,$ and $\C=(\I-\W)^2$.
\smallskip 

\noindent {\bf 4)  General primal-dual scheme~\cite{jakovetic2018unification,mansoori2019general}}: A general distributed primal-dual algorithm was proposed in~\cite{jakovetic2018unification} for \eqref{prob:dop_nonsmooth_same}  with $G=0$ as follows
\begin{subequations}\label{alg:Jakovetic}
\begin{align}
\x^{k+1}&=\W\x^k-\gamma(\nabla f(\x^k)+Y^k),\label{alg:Jakovetic_x}\\
Y^{k+1}&=Y^k-(\I-\W)(\nabla f(\x^k)+Y^k-\B'\x^k),\label{alg:Jakovetic_y}
\end{align}
\end{subequations}
where $\B'$ can be $b\I$ or $b\W$ for some positive constant $b>0$ therein. Eliminating the $Y$-variable,  \eqref{alg:Jakovetic} reduces to
\[
\begin{aligned}
\x^{k+2}&=2\W\x^{k+1}-(\W^2+\gamma(\I-\W)\B')\x^k\\
&-\gamma(\nabla f(\x^{k+1})-\nabla f(\x^k)),
\end{aligned}
\]
which corresponds to the proposed algorithm, with $\A=\W^2+\gamma(\I-\W)\B',\B=\I,\C=(\I-\W)^2+\gamma(\I-\W)\B'$. 
\\\indent Similarly, building on a general augmented Lagrangian, another general primal-dual algorithm was proposed in~\cite{mansoori2019general} for \eqref{prob:dop_nonsmooth_same}  with $G=0$, which reads
\begin{subequations}\label{alg:ermin}
\begin{align}
\x^{k+1}&=(\I-\alpha\B')^K\x^k-\alpha\C'(\nabla f(\x^k)+\A'^\top Y^k), \label{alg:ermin_x}\\
Y^{k+1}&=Y^k+\beta\A'\x^{k+1}, \label{alg:ermin_y}
\end{align}
\end{subequations}
where $\A',\B',\C'$ are certain weight matrices therein and $\C'=\sum_{i=0}^{K-1}(\I-\alpha\B')^i$, with $K$ being the number of communication steps performed at each iteration. Eliminating $\y$ yields
\[
\begin{aligned}
\x^{k+2}&=(\I+(\I-\alpha\B')^K-\alpha\beta\C'\A'^\top\A')\x^{k+1}\\
&~~~-(\I-\alpha\B')^K\x^k-\alpha\C'(\nabla f(\x^{k+1})-\nabla f(\x^k))
\end{aligned}
\]
which corresponds to Algorithm \eqref{alg:g-ABC} with $\A=(\I-\alpha\B')^K,\B=\C',\C=\alpha\beta\C'\A'^\top\A'$. Notice that, letting $\W=\I-\alpha\B'$ and $\B'=\beta\A'^\top\A'$, we have $\A=\W^K,\B=\sum_{i=0}^{K-1}\W^i$ and $\C=(\I-\W)\sum_{i=0}^{K-1}\W^i=\I-\W^K$, which satisfy Assumption~\ref{assum:cond_A_B}.\smallskip 

\noindent {\bf 6) Decentralized proximal algorithm~\cite{alghunaim2019linearly}}: A proximal algorithm is proposed to solve~\eqref{prob:dop_nonsmooth_same} with $G\neq 0$, which reads
\[
\begin{aligned}
\z^{k+2}&=(\I-\alpha\B')\z^{k+1}+(\I-\B')(\x^{k+1}-\x^k)\\
&-\gamma(\nabla f(\x^{k+1})-\nabla f(\x^k)).
\end{aligned}
\]
where $\x^k=\prox{\gamma g}{\z^k}$ and $\B'$ is some properly chosen  matrix that ensures consensus. It is easy to show that the above algorithm corresponds to Algorithm \eqref{alg:g-ABC} with $\A=\I-\B',\, \B=\I,\, \C=\alpha\B'$. Choosing $\W=\I-\B'$, we have $\A=\W,\B=\I$ and $\C=\alpha(\I-\W)$, which clearly satisfy Assumption~\ref{assum:cond_A_B}.  {Note that, since  $B=I$, this algorithm  (and thus \cite{alghunaim2019linearly}) is of CTA form and cannot model  ATC-based schemes, such as NEXT/AugDGM and NIDS/Exact Diffusion listed in Table~\ref{tab:instances}.}\vspace{-0.2cm}
 
\section{An Operator Splitting Interpretation}
Our convergence analysis builds on an equivalent fixed-point reformulation of Algorithm \eqref{alg:g-ABC}, whose mapping enjoys a favorable decomposition in terms of contractive and nonexpansive operators.  We begin introducing the following assumptions. 
\begin{assum}\label{assum:B_C_commute}
The weight matrices satisfy: \begin{itemize}
  \item[i)] $A = BD$; 
  \item[ii)] $B$ and $C$ commute.
\end{itemize}
\end{assum}
Under the above assumption, the following lemma provides an operator splitting form for Algorithm~\eqref{alg:g-ABC}.

\begin{prop}\label{lm:transform}
Given the sequence $\{(\z^k, \x^k, \y^k)\}_{k\in \mathbb{N}_+}$ generated by Algorithm \eqref{alg:g-ABC}, define $U^k \triangleq 
[({\z}^k)^\top, ({\y}^k)^\top]^\top$. Under Assumption~\ref{assum:B_C_commute}, the following hold:

\noindent\textbf{1)}\vspace{-0.4cm}
 \begin{equation} \label{eq:def_U_tilde}U^{k}=
\begin{bmatrix}B & 0\\
0 & B\sqrt{C}
\end{bmatrix}  \widetilde{U}^k, \quad\text{with}\quad \widetilde{U}^k\triangleq \begin{bmatrix}\widetilde{\z}^{k}\\
\sqrt{C}\widetilde{\y}^{k}
\end{bmatrix};\vspace{-0.2cm}\end{equation}
and $\{\widetilde{U}^k\}_k$ satisfies the following dynamics
\begin{equation}\label{eq:dyn_U_tilde}
\widetilde{U}^{k+1} = \underbrace{
\begin{bmatrix}
(D- \gamma \nabla f)\circ \text{prox}_{\gamma g} \circ B  & -\sqrt{C}\\
\sqrt{C}(D- \gamma \nabla f)\circ \text{prox}_{\gamma g} \circ B   & I-C
\end{bmatrix}}_{T}
\widetilde{U}^k,\quad k\geq 1,
\end{equation}
 with initialization  $\widetilde{\z}^{1} = \widetilde{\y}^{1} = (D - \gamma \nabla f) (\x^{0})$;
 
 \noindent\textbf{2)} The operator $T$ can be decomposed as
\begin{equation}\label{eq:T_operator}
T=
\underbrace{
\begin{bmatrix}
I  & -\sqrt{C}\\
\sqrt{C}   & I-C
\end{bmatrix}}_{\triangleq T_C}
\underbrace{
\begin{bmatrix}
D- \gamma \nabla f & 0\\
0  & I
\end{bmatrix}}_{\triangleq T_f}
\underbrace{
\begin{bmatrix}
\text{prox}_{\gamma g}  & 0 \\
0  & I
\end{bmatrix}}_{\triangleq T_g}
\underbrace{
\begin{bmatrix}
B &  0 \\
0    & I
\end{bmatrix}}_{\triangleq T_B},
\end{equation} where   $T_C$ and $T_B$ are the operators associated with communications while    $T_f$   and $T_g$ are the   gradient and proximal operators, respectively;

\noindent\textbf{3)} Every fixed point $\widetilde{U}^\star\triangleq [\widetilde{Z}^\star,\sqrt{C}\widetilde{Y}^\star]$ of $T$ is such  that $\x^\star \triangleq \text{prox}_{\gamma g}(B \widetilde{Z}^\star)\in \mathcal{S}_{\texttt{Fix}}$. Therefore,  $\x^\star  =\ones x^{\star \top}$, where  $x^\star$ is an optimal solution of  \eqref{prob:dop_nonsmooth_same}.
\end{prop}
\begin{proof}
From \eqref{alg:g-ABC}, we have
$\z^{k+1} =(\I-\C)\z^{k} 
  +\A(\x^{k}-\x^{k-1})-\gamma\B(\nabla f(\x^{k})-\nabla f(\x^{k-1}))$, which applied recursively yields
\begin{align*}
& \z^{k+1} \\
=    &  \sum_{t=1}^k (\I-\C)^{k-t} \left( \A(\x^{t}-\x^{t-1})-\gamma\B(\nabla f(\x^{t})-\nabla f(\x^{t-1}) \right)  \\
& + (\I-\C)^k  \left( \A\x^{0} - \gamma\B \nabla f(\x^0) \right)  \\
\stackrel{(*)}{=}    &  B \bigg(\sum_{t=1}^k (\I-\C)^{k-t} \big( D(\x^{t}   -\x^{t-1})-\gamma(\nabla f(\x^{t})-\nabla f(\x^{t-1}) \big)  \\
& + (\I-\C)^k  \left( D\x^{0} - \gamma \nabla f(\x^0) \right)   \bigg)\\
= & B \sum_{t=0}^k (\I-\C)^{k-t}  (D - \gamma \nabla f) (\x^{t}) \\
&  - B\sum_{t=0}^{k-1} (\I-\C)^{k-1-t}  (D - \gamma \nabla f) (\x^{t}) ,
\end{align*}
where in $(*)$ we used Assumption \ref{assum:conditions_convergence}i) and  \ref{assum:conditions_convergence}iv). 

Define $\widetilde{\z}^k$ such that $\z^k = B \widetilde{\z}^k,$   $k \geq 1$; and let 
\begin{align}\label{eq:tildeY}
\widetilde{\y}^{k+1} \triangleq \sum_{t=1}^{k+1} \widetilde{\z}^t = \sum_{t=0}^k (\I-\C)^{k-t}  (D - \gamma \nabla f) (\x^{t}),
\end{align}
for  $k\geq 0$. It is clear from the definition of $\widetilde{\z}$ and $\widetilde{\y}$ that
\begin{align}\label{eq:dym_tilde_system}
\begin{bmatrix}
\widetilde{\z}^{k+1}\\
\widetilde{\y}^{k+1}
\end{bmatrix} =
\begin{bmatrix}
(D- \gamma \nabla f)\circ \text{prox}_{\gamma g} \circ B  & -C\\
(D- \gamma \nabla f)\circ \text{prox}_{\gamma g} \circ B   & I-C
\end{bmatrix}
\begin{bmatrix}
\widetilde{\z}^{k}\\
\widetilde{\y}^{k}
\end{bmatrix}.
\end{align}

Introducing  $\widetilde{U}^k$ as defined in \eqref{eq:def_U_tilde}, it follows from (\ref{eq:dym_tilde_system}) that  $\widetilde{U}^k$ obeys the dynamics  \eqref{eq:dyn_U_tilde}.  The equation $Y^k = BC \widetilde{\y}^k $ follows readily from \eqref{alg:g-ABC_y} and \eqref{eq:tildeY}. Finally, the decomposition of the transition matrix $T$ can be checked by  inspection.  

We prove now the last statement of the theorem. For every fixed point $\widetilde{U}^\star\triangleq [\widetilde{Z}^\star,\sqrt{C}\widetilde{Y}^\star]$ of $T$, we have\footnote{For any matrix $M$, we use $\Span{M}$ to denote its column space.} $\Span{\widetilde{\z}^\star} \subset \Span{1}$ and 
\begin{align}\label{eq:fix1}
-1^\top \bracket{B(D - \gamma \nabla f)\circ \text{prox}_{\gamma g}\circ B \bracket{\widetilde{Z}^\star}} + 1^\top B \widetilde{Z}^\star =0.
\end{align}
For $\x^\star \triangleq \text{prox}_{\gamma g}(B \widetilde{Z}^\star)$, it holds $\Span{\x^\star} \subset \Span{1}$ and
\begin{align}\label{eq:fix2}
B \widetilde{Z}^\star \in \x^\star + \gamma \partial g(\x^\star ).
\end{align}
Combining \eqref{eq:fix1} and \eqref{eq:fix2} leads to
$
1^\top(I-\A)\x^\star+\gamma \, 1^\top\B\nabla f(\x^\star)\in - \gamma\,  1^\top\partial g(\x^\star), 
$
which is equivalent to $\x^\star \in \mathcal{S}_{\texttt{Fix}}$. The proof follows from Lemma~\ref{lemma_eq_KKT} and \ref{lem:iff_fixed_point}.\smallskip 
\end{proof}



We summarize next the main properties of the operators  $T_C$,  $T_f$,  $T_g$,  and  $T_B$, which will be instrumental to establish linear convergence rate of the proposed algorithm.  We will use the following notation: given  $X \in \mathbb{R}^{2m\times d}$, we denote by  $(X)_u$ and $(X)_\ell$  its upper and lower $m\times d$ matrix-block; for any matrix $A\in \mathbb{R}^{m\times m}$, we denote $\Lambda_{A} = \diag(A, I)\in \mathbb{R}^{2m \times 2m}$ and $V_A = \diag(I,A)\in \mathbb{R}^{2m \times 2m}.$ 

\begin{lem}[Contraction of $\T_C$]
\label{lem:contraction_T_c}
The operator $\T_C$ satisfies
\[
\norm{\T_C\, X-\T_C \,Y }_{\Lambda_{I-C}} = \norm{X-Y }_{V_{I-C}},\quad \forall X,Y\in\mathbb{R}^{2m\times d}.
\]
\end{lem}
\begin{proof}
The result comes readily from the definition of $T_C$ and the fact that $T_C^\top\Lambda_{I-C} \T_C=V_{I-C}$.
\end{proof} 

\begin{lem}[Contraction of $\T_f$]
\label{lem:contraction_T_f}
Consider the operator $\T_f$ under  Assumption~\ref{assum:smooth_strong_reg}, with $\mu>0$, and  $0\prec \D \preceq I$.
If $0<\gamma\leq  \gamma^\star (D)$ with\vspace{-0.2cm}
\begin{equation}\label{eq:opt_ss_D}
\gamma^\star(D) \triangleq \frac{2\lambda_{\min}(D)}{L+\mu \cdot \lambda_{\min}(D)},\vspace{-0.2cm}
\end{equation}
 then    \[
\begin{aligned}
&\norm{(\T_f\x)_u-(\T_f\y)_u}\leq q(D,\gamma)\norm{(\x)_u-(\y)_u}_D,
\end{aligned}
\]
 $\forall  \x,\y \in\mathbb{R}^{2m\times d}$, where  \vspace{-0.2cm}
\begin{equation}\label{eq:contract_factor_T_f}
q(D,\gamma)= 1-   \frac{2\gamma L}{\kappa + \lambda_{\min}(D)}.
\end{equation}
The  stepsize minimizing the contraction factor is $\gamma=  \gamma^\star(D)$, resulting in  the smallest achievable  $q(D,\gamma)$, given by 
\begin{equation}\label{eq:opt_factor_T_f}
q^\star(D) \triangleq \frac{\kappa-\lambda_{\min}(D)}{\kappa+\lambda_{\min}(D)}.
\end{equation}
\end{lem}
\begin{proof}
See Sec. \ref{sec:pf_tf} in Appendix.
\end{proof}

We conclude with the properties of $\T_g$ and $\T_B$, which follow  readily from the non-expansive property of the proximal operator and the linear nature of $\T_B$, respectively. 
\begin{lem}[Non-expansiveness  of $T_g$]
 The operator $\T_g$ satisfies:  $\forall X, Y\in\mathbb{R}^{2m\times d}$,
\label{lem:contraction_T_g}
\begin{align*}
  \norm{(\T_g\, X)_u-(\T_g\, Y)_u}^2&\leq \norm{(X)_u-(Y)_u}^2 \\
  (\T_g\,X)_\ell &= (X)_\ell.
\end{align*}
\end{lem}
\begin{lem} [Non-expansiveness   of $T_B$]
The operator  $\T_B$ satisfies: $\forall X \in\mathbb{R}^{2m\times d}$,
\label{lem:contraction_T_B}
\begin{align*}
\norm{(\T_B\, X)_u}^2  = \norm{(X)_u}_{B^2}^2, \quad (\T_g\,X)_\ell  = (X)_\ell.
\end{align*}
\end{lem}

\section{Linear Convergence} 
In this section we prove linear convergence of Algorithm \eqref{alg:g-ABC}, under strong convexity of each $f_i.$ 
Since most of the algorithms in the literature considered only the case $G=0$, we begin with that setting (cf. Sec. \ref{sec:linear-rate-G_zero} ). Sec.\ref{sec:sc_G} extends our analysis to    $G\neq 0$. Finally, we comment our results  in Sec.\ref{sec:discussion}. \vspace{-0.3cm}

\subsection{Convergence under $G= 0$} \label{sec:linear-rate-G_zero} Consider Problem \eqref{prob:dop_nonsmooth_same} with  $G=0$. Algorithm~\eqref{alg:g-ABC} reduces to\vspace{-0.2cm}
\begin{subequations}\label{alg:ABC_no_G}
\begin{align}
X^{k+1}&= AX^k-\gamma B\nabla f(X^k)-Y^k, \label{eq:ABC_alg_equi_X}\\
Y^{k+1}&=Y^k+CX^{k+1}, \label{eq:ABC_alg_equi_Y}
\end{align}
\end{subequations}
with $X^0 \in \mathbb{R}^{m\times d}$ and $Y^0 = 0$.    

Theorem~\ref{thm:contraction_T_c_T_f_T_B} below establishes linear convergence   of  Algorithm~\eqref{alg:ABC_no_G} under the following assumption on $A, B,$ and $C$. 

\begin{assum}\label{assum:conditions_convergence_no_reg} 
The weight matrices $A\in \mathbb{R}^{m\times m}$, $B,\,C\in \mathbb{S}^{m}$ and the stepsize $\gamma$ satisfy: \begin{itemize} 
  \item[i)] $A=BD$, with $0 \prec D \preceq I$ and $D\in \mathbb{S}^{m}$; 
  \item[ii)] $\ones^\top D \ones = m$ and $\ones^\top B = \ones^\top$;   \item[iii)]  $0 \preceq C \prec I$ and  $\Null{{C}}=\Span{\ones}$;
  \item[iv)] $B$ and $C$ commute;
  \item[v)]    ${q(D,\gamma)}^2 AB \prec (I-C)$ and $ 0<\gamma \leq \gamma^\star(D)$,
\end{itemize}
where $q(D,\gamma)$ and $\gamma^\star(D)$ are defined in  \eqref{eq:contract_factor_T_f} and (\ref{eq:opt_ss_D}), respectively.
\end{assum}

Assumption \ref{assum:conditions_convergence_no_reg} 
is quite mild and satisfied by a variety of algorithms; for instance, this is the case for all the schemes  in Table~\ref{tab:instances}. In particular, the commuting  property of $B$ and $C$ is trivially satisfied when $B, C\in P_K(W)$, for some given $W\in \mathcal{W}_{\mathcal G}$ (as in Table \ref{tab:instances}). Also, one can  show that condition v)  in Assumption~\ref{assum:conditions_convergence_no_reg} is {\it necessary} to achieve linear rate.

\begin{thm}[Linear rate for $T_CT_fT_B$]
\label{thm:contraction_T_c_T_f_T_B}
Consider Problem~\eqref{prob:dop_nonsmooth_same} under Assumption~\ref{assum:smooth_strong_reg}, $\mu>0$, and $G=0$, with   solution    $x^\star$. Let $\{(\x^k,Y^k)\}_{k \in \mathbb{N}_+}$ be the sequence generated by Algorithm \eqref{alg:ABC_no_G} under Assumption~\ref{assum:conditions_convergence_no_reg}. Then,   $\norm{\x^k-\ones x^{\star\top}}^2=\mathcal{O}(\delta^k),$ with
\begin{equation} \label{eq:def_lambda_Th15}  
 \delta \triangleq  \max\Big(q(D,\gamma)^2 \lambda_{\max} (A B(I-C)^{-1}), 1-\lambda_2(C) \Big), \end{equation}
where $q(D,\gamma)$ is defined in~\eqref{eq:contract_factor_T_f}.
\end{thm}\begin{proof} 
Since \eqref{alg:ABC_no_G} corresponds to Algorithm \eqref{alg:g-ABC} with $G=0$, by Assumption~\ref{assum:conditions_convergence_no_reg} and Prop.~\ref{lm:transform}, \eqref{alg:ABC_no_G} can be equivalently rewritten in the form \eqref{eq:dyn_U_tilde}, with $T_g=I$;  and thus the $Z$- and $X$-variables coincide. Define $\x^\star=\z^\star\triangleq \ones x^{\star\top}$.  
Let $\widetilde{U}^k={[(\widetilde{\z}^k)^\top,(
\sqrt{C}\widetilde{\y}^k)^\top]}^\top$ be the auxiliary sequence defined in~\eqref{eq:def_U_tilde} with $\widetilde{U}^\star\triangleq [\widetilde{Z}^\star,\sqrt{C}\widetilde{Y}^\star]$ the fixed point of $T=T_CT_fT_B$. Then, we have\vspace{-0.1cm}
\begin{equation}\label{eq:connect_R_rate}
\begin{aligned}
& \norm{\x^k-\x^\star}^2 =\norm{\z^k-\z^\star}^2 \stackrel{\eqref{eq:def_U_tilde}}{\leq} \norm{\widetilde{\z}^k-\widetilde{\z}^\star}_{\B^2}^2 \\
& \leq \frac{\lambda_{\max}(B^2)}{\lambda_{\min}(I-C)} \norm{\widetilde{\z}^k-\widetilde{\z}^\star}_{I-C}^2  \leq \frac{\lambda_{\max}(B^2)}{\lambda_{\min}(I-C)} \norm{\widetilde{\U}^k-\widetilde{\U}^\star}^2_{\Lambda_{I-C}}.
\end{aligned}
\end{equation}
Using  \eqref{eq:dyn_U_tilde} in (\ref{eq:connect_R_rate}), it is sufficient  to prove that $T$ is contractive w.r.t. the norm $\norm{\cdot}_{\Lambda_{I-C}}$. 
To this end, consider the following chain of inequalities: $\forall \, X, Y\in\mathbb{R}^{2m\times d}$,  $X_\ell,Y_\ell \in \Span{\sqrt{C}}$,
\begin{align*}
& \norm{T\,X- T\,Y}_{\Lambda_{I-C}}^2 \\
&= \norm{T_C \circ T_f \circ T_B\,(X)- T_C \circ T_f \circ T_B\,(Y)}_{\Lambda_{I-C}}^2 \\
& \stackrel{Lem.~\ref{lem:contraction_T_c}} {=}\norm{T_f \circ T_B\,(X)- T_f \circ T_B\,(Y)}_{V_{I-C}}^2 \\
& \stackrel{Lem.~\ref{lem:contraction_T_f}}{\leq} \norm{T_B\,(X)- T_B\,(Y)}_{V_{I-C} \Lambda_{{q(D,\gamma)}^2 \,D}}^2 \\
& \stackrel{Lem.~\ref{lem:contraction_T_B}}{=} \norm{X- Y}_{V_{I-C} \,\Lambda_{{q (D,\gamma)}^2\, BDB} }^2 \stackrel{(*)}{\leq} \delta \, \norm{X- Y}_{\Lambda_{I-C}}^2,
\end{align*}
where (*) is due to: i)   for all $\,(Z)_u \in \mathbb{R}^{m\times d}$,
 \begin{align*} 
& \|(Z)_u\|^2_{BDB}=\|(I-C)^{\frac{1}{2}}(Z)_u\|^2_{(I-C)^{-1/2} BDB(I-C)^{-1/2}}   \\
& \leq  \lambda_{\text{max}} (A B(I-C)^{-1})\|(I-C)^{\frac{1}{2}}(Z)_u\|^2 \\
 & = \lambda_{\text{max}} (A B(I-C)^{-1})\norm{(Z)_u}^2_{I-C};
\end{align*}
 and ii) $X_\ell,Y_\ell \in \Span{\sqrt{C}}$. 
\end{proof}

Note that Theorem \ref{thm:contraction_T_c_T_f_T_B} is the first unified convergence result stating linear rate  for  ATC (corresponding to $D=I$) {\it and} CTA (corresponding to $B=I$)  schemes. Because of this generality and consistently with existing conditions for the convergence of CTA-based schemes,    the choice of the stepsize satisfying Assumption \ref{assum:conditions_convergence_no_reg} might depend on some network parameters. This is due to the fact that $\lambda_{\max} (AB(I-C)^{-1}) \geq 1$, since  $(I-C)^{-1/2} AB(I-C)^{-1/2} \ones = \ones$. Hence, when  $\lambda_{\max} (AB(I-C)^{-1}) > 1$, the stepsize needs to be leveraged to guarantee that $q(D,\gamma)^2 \lambda_{\max} (A B(I-C)^{-1})<1$, reducing the  range of feasible values.  For instance, this happens for i) CTA schemes  ($B=I$) such that  $D \preceq I-C$ does not hold; of ii) for ATC schemes ($D=I$) that do not satisfy the condition $B^2 \preceq I-C$.

 The following corollary provides a condition on the weight matrices enlarging the range of the stepsize to $[0,\gamma^\star(D)]$. Furthermore, the tuning minimizing the contraction factor $\delta$ in  \eqref{eq:def_lambda_Th15} is derived. 
   
\begin{col}\label{col:optimal_D}
Consider the setting of Theorem~\ref{thm:contraction_T_c_T_f_T_B}, and  further assume $AB \preceq I-C$.  
Then, $\norm{\x^k-\ones x^{\star\top}}^2=\mathcal{O}(\delta^k),$ with\vspace{-0.1cm}
\begin{equation} \label{eq:def_lambda_simp}
\delta =  \max \Big(q(D,\gamma)^2, 1-\lambda_2(C) \Big). 
\end{equation}
The stepsize that minimizes (\ref{eq:def_lambda_simp}) is $\gamma=\gamma^\star(D)=\frac{2\lambda_{\min}(D)}{L+\mu \cdot \lambda_{\min}(D)}$, 
resulting in the contraction factor
\begin{equation} \label{eq:def_lambda_simp_2}
\delta =  \max \left(\left(\frac{\kappa-\lambda_{\min}(D)}{\kappa+\lambda_{\min}(D)}\right)^2, 1-\lambda_2(C) \right). 
\end{equation}
The smallest   $\delta$ is achieved choosing  $\D=\I$, which yields    $\gamma = \gamma^\star \triangleq \frac{2}{\mu+L}$ and \vspace{-0.3cm}
\begin{equation}\label{rate-separation}
\delta^\star=\max\left\{ \bracket{\frac{\kappa-1}{\kappa+1}}^2, ~ 1-\lambda_2(\C) \right\}.
\end{equation}
\end{col}

\begin{proof}
Since $(I-C)^{-1/2} AB (I-C)^{-1/2} \ones= \ones$ and $AB \preceq I-C,$ we have $\lambda_{\max} (AB(I-C)^{-1}) =1$, which together with \eqref{eq:def_lambda_Th15} yield  
  \eqref{eq:def_lambda_simp}. Eq. \eqref{eq:def_lambda_simp_2}
 follows readily from the decreasing property of $q(D,\gamma)$ on $\gamma\in (0,\gamma^\star(D)]$, for any given $0\prec D\preceq I$. Finally,  \eqref{rate-separation} is the result of the following optimization problem: $\max_{D\in \mathbb{S}^{m}} \lambda_{\min}(D)$, subject to  $0\prec \D\preceq\I$ [Assumption \ref{assum:conditions_convergence_no_reg}(i)]  and  $\ones^\top\D \ones=m$ [Assumption \ref{assum:conditions_convergence_no_reg}(ii)], whose solution is $D=I$. 
\end{proof}
\vspace{-0.4cm}

\subsection{The general case $G\neq 0$}\label{sec:sc_G}
We establish now  linear convergence of   Algorithm \eqref{alg:g-ABC} applied to Problem \eqref{prob:dop_nonsmooth_same}, with $G\neq 0$. 
 {We introduce  the following assumption similar to Assumption~\ref{assum:conditions_convergence_no_reg} for $G = 0$.}
\begin{assum}\label{assum:conditions_convergence} 
The weight matrices $A\in \mathbb{R}^{m\times m},\,  B,\,C\in \mathbb{S}^{m}$ and the stepsize $\gamma$ satisfy: \begin{itemize} 
  \item[i)] $A=BD$ with $0 \prec D \preceq I$ and $D\in \mathbb{S}^{m}$; 
  \item[ii)] $\ones^\top D \ones = m$ and $\ones^\top B = \ones^\top$; 
  \item[iii)] $0 \preceq C\prec I$ and $\Null{{C}}=\Span{\ones}$;
  \item[iv)] $B$ and $C$ commute;
  \item[v)]   ${q(D,\gamma )}^2 B^2 \prec (I-C)$    and   $0<\gamma\leq  \gamma^\star(D)$,
\end{itemize}
where   $q(D,\gamma)$ and $\gamma^\star(D)$ are defined in  \eqref{eq:contract_factor_T_f} and (\ref{eq:opt_ss_D}), respectively.
\end{assum}

Condition  v) in Assumption \ref{assum:conditions_convergence} is slightly stronger than its counterpart in Assumption \ref{assum:conditions_convergence_no_reg} (as $BDB\prec B^2$). This is due to the complication of dealing with the nonsmooth function $G$ (the presence of the proximal operator $T_g$).  However, as shown in Corollary \ref{col:contraction_overall_optimal_choice} below, this does not affect the smallest achievable contraction rate, which coincides with the one attainable when $G=0$.   {Note that Assumption~\ref{assum:conditions_convergence} is satisfied by all the algorithms in Table \ref{tab:instances}.}



\begin{thm}[Linear rate for $T=T_C T_f T_g T_B$]
\label{thm:contraction_T_c_T_f}
 
  Consider Problem~\eqref{prob:dop_nonsmooth_same} under Assumption~\ref{assum:smooth_strong_reg} with $\mu>0$, whose optimal solution  is $x^\star$. Let $\{(\x^k,\z^k,Y^k)\}_{k\geq 0}$ be the sequence generated by Algorithm \eqref{alg:g-ABC} under Assumption~\ref{assum:conditions_convergence}. 
Then $\norm{\x^k-\ones x^{\star\top}}^2=\mathcal{O}(\delta^k)$, with    
\begin{equation} \label{eq:def_lambda_reg}
\begin{aligned}
\delta \triangleq  \max \left({q(D,\gamma)}^2\lambda_{\text{max}} (B^2(I-C)^{-1}), ~1-\lambda_2(C) \right), 
 \end{aligned}
\end{equation}
where $q(D,\gamma)$ is defined in~\eqref{eq:contract_factor_T_f}. 
\end{thm}
 {The proof of Theorem~\ref{thm:contraction_T_c_T_f} is similar to that of Theorem~\ref{thm:contraction_T_c_T_f_T_B} and is  provided in the  Appendix.}

\begin{col}
\label{col:contraction_overall_optimal_choice}
Consider the setting of Theorem~\ref{thm:contraction_T_c_T_f}, and further assume  $B^2\preceq I-C$.  Then, the same conclusions 
as in   Corollary \ref{col:optimal_D}  hold for Algorithm \eqref{alg:g-ABC}.\vspace{-0.3cm} 
\end{col}

\subsection{Discussion}\label{sec:discussion} 

\subsubsection{Unified convergence conditions}  Theorems~\ref{thm:contraction_T_c_T_f_T_B} and \ref{thm:contraction_T_c_T_f}    
offer a unified platform for the analysis and design of a gamut of linearly convergence algorithms--all the schemes, new and old,  that can be written in the form \eqref{alg:ABC_no_G} and \eqref{alg:g-ABC}  satisfying Assumption \ref{assum:conditions_convergence_no_reg} and \ref{assum:conditions_convergence}, respectively. For instance, our convergence results 
embrace   {\it both} ATC and CTA algorithms, solving  either smooth ($G=0$)  or {\it composite} ($G\neq 0$) optimization problems. This improves on \cite{alghunaim2019linearly} and \cite{alghunaim2019decentralized} 
and contrasts the majority of the literature, wherein  proposed algorithms have been  generally studied in isolation, resulting in ad-hoc convergence conditions and rates. 
Our  results are instead  widely applicable--e.g., to all the algorithms listed in Table~\ref{tab:optim_linear_rate}--and tighter than existing rate expressions; see  Sec.~\ref{subsec_comparison}. \smallskip 

\subsubsection{On the rate expression}

 
We comment  the  expression of  the rate  focusing on Theorem \ref{thm:contraction_T_c_T_f}   and Corollary \ref{col:contraction_overall_optimal_choice} ($G\neq 0$); same conclusions can be drawn for Algorithm \eqref{alg:ABC_no_G} (Theorem \ref{thm:contraction_T_c_T_f_T_B}  and Corollary \ref{col:optimal_D}).   Theorem \ref{thm:contraction_T_c_T_f}    provides the explicit expression of the linear rate provably achievable by Algorithm \eqref{alg:g-ABC}, for a given choice of  the weight matrices $A$, $B$, and $C$ and stepsize $\gamma$ (satisfying Assumption \ref{assum:conditions_convergence}).   In general,   this rate depends on both optimization parameters ($L$ and $\mu$) and network-related quantities ($A$, $B$, and $C$); furthermore,  feasible stepsize values and network parameters are coupled by Assumption \ref{assum:conditions_convergence}v). {\bf CTA-based schemes:} This is consistent with existing convergence results of CTA-based algorithms (known only for $G=0$), which are special cases of Algorithm \eqref{alg:ABC_no_G}. For instance, consider  EXTRA~\cite{shi2015extra} and DIGing~\cite{nedich2016achieving} (corresponding to Algorithm \eqref{alg:ABC_no_G} with  $B=I$, cf. Table~\ref{tab:optim_linear_rate}): $\gamma$,  $C$ and $D$ are coupled via   the condition ${q(D,\gamma )}^2\prec (I-C)$, instrumental to achieve linear rate. 
{\bf ATC-based schemes:} For algorithms in the ATC form, i.e.,   $A=B$, less restrictive conditions are required.  For instance, when Assumption \ref{assum:conditions_convergence}v) is satisfied by $B^2\prec I-C$--a condition that is met by several algorithms in Table~\ref{tab:optim_linear_rate}--the stepsize can be chosen  in the larger region $[0,\gamma^\star(D)]$, 
resulting in the smaller rate  $\max(q(D,\gamma),1-\lambda_2(C))\geq \max(q^\star(D),1-\lambda_2(C))$ (recall that, in such a case,  $\lambda_{\text{max}} (B^2(I-C)^{-1})=1$), where the lower bound is achieved when $\gamma=\gamma^\star(D)$ (cf. Corollary \ref{col:contraction_overall_optimal_choice}).

 On the other hand, when the algorithm parameters can be freely designed, Corollary \ref{col:optimal_D} offers the ``optimal'' choice, resulting in the smallest contraction factor, as in (\ref{rate-separation}).  
   This  instance enjoys two desirable properties, namely: 
   
 \textbf{(i) Network-independent stepsize:} The    stepsize $\gamma^\star$ in Corollary \ref{col:optimal_D} does not depend on the network parameters but only on the optimization and its value coincides with the optimal stepsize of the centralized proximal-gradient algorithm. This is a major advantage over current distributed schemes applicable to \eqref{prob:dop_nonsmooth_same} (but with $G\neq 0$) and complements the results in  \cite{li2017decentralized}, whose algorithm  however  cannot deal with  the non-smooth term  $G$ and use more stringent stepsize.
  
 \textbf{(ii) Rate-separation:}  The   rate (\ref{rate-separation}) is determined by the worst rate  between the one  due to the communication $[1-\lambda_2(\C)]$ and that of the optimization $[((\kappa-1)/(\kappa+1))^2]$. This separation is the key enabler   for our distributed scheme to achieve the convergence rate of the centralized proximal gradient algorithm-we elaborate on this property  next. 

   \subsubsection{Balancing computation and communications}\label{sec:tradeoff} Note  that    $\rho_{\texttt{opt}}\triangleq (\kappa-1)/(\kappa+1)$  is the  rate of the centralized proximal-gradient algorithm applied to   \eqref{prob:dop_nonsmooth_same}, under Assumption 1.
This means that if the network is ``sufficiently connected'', specifically   $1-\lambda_2(\C)\leq \rho_{\texttt{opt}}^2$, the proposed algorithm   converges at  the \emph{desired}  linear rate $\rho_{\texttt{opt}}$. 
 On the other hand, when  
$1-\lambda_2(\C)>\rho_{\texttt{opt}}^2$, one can still achieve the centralized rate $\rho_{\texttt{opt}}$   by enabling multiple (finite) rounds of communications per proximal gradient evaluations. 
 Two strategies are:  1) performing multiple rounds of  consensus using each time the same weight matrix; or 2) employing acceleration via Chebyshev polynomials. 
\textbf{1) Multiple rounds of consensus:}
Given a weight matrix $W\in \mathcal{W}_{\mathcal{G}}$, as concrete example,  consider the case $W\in\mathbb{S}^m_{++}$ and    $\A=\B=\I-\C=W^K$, with $K\geq 1$, which implies $\B^2\preceq \I-\C$ (cf. Corollary~\ref{col:optimal_D}).  The resulting algorithm   will require 
$K$ rounds of  communications  (each of them using  $\W$) per gradient evaluation. Denote 
$\rho_{\texttt{com}}\triangleq\lambda_{\max}(\W-\J)$; we have $1-\lambda_2(\C)=\lambda_{\max}(W^K-\J)=\rho^K_{\texttt{com}}$. The value of $K$ is chosen to minimize the resulting rate $\lambda$ [cf. (\ref{rate-separation})], i.e., such that $\rho_{\texttt{com}}^K\leq\rho_{\texttt{opt}}^2$, which leads to $K=\lceil\log_{\rho_{\texttt{com}}}({\rho^2_{\texttt{opt}}})\rceil$.  
\noindent \textbf{2) Chebyshev acceleration:}
To further reduce the  communication cost, we can leverage Chebyshev acceleration~\cite{auzinger2011iterative}. As specific example, consider the case $\W\in\mathbb{S}^m$ is invertible;  we set   $\A=\P_K(\W)$ and $P_K(1)=1$ (the latter is to ensure the double stochasticity of $\A$), with $P_K\in\mathbb{P}_K$. 
This leads to $1-\lambda_2(C)=\lambda_{\max}(\A^2-\J)$.  {The idea of Chebyshev acceleration is to find the ``optimal'' polynomial $P_K$ such that $\lambda_{\max}(\A^2-\J)$ is minimized, 
 i.e.,
$\rho_C\triangleq\min_{P_K\in \mathbb{P}_K,P_K(1)=1}\max_{t\in[-\rho_{\texttt{com}},\rho_{\texttt{com}}]} \abs{P_K(t)}$.
The optimal solution of this problem  is  $P_K(x)={T_K(\frac{x}{\rho_{\texttt{com}}})}/{T_K(\frac{1}{\rho_{\texttt{com}}})}$    \cite[Theorem 6.2]{auzinger2011iterative}, with $\alpha'=-\rho_{\texttt{com}}$, $\beta'=\rho_{\texttt{com}},\gamma'=1$ (which are certain parameters therein),
where $T_K$ is the $K$-order Chebyshev polynomials that can be computed in a distributed manner via the following  iterates \cite{auzinger2011iterative,scaman17optimal}:
$T_{k+1}(\xi) =2\xi \, T_k(\xi)-T_{k-1}(\xi),$  $k\geq 1$,
with $T_0(\xi)=1$, $T_1(\xi)=\xi$. 
  Also, invoking \cite[Corollary 6.3]{auzinger2011iterative}, we have
$\rho_C=\frac{2c^K}{1+c^{2K}}$,
where $c=\frac{\sqrt{\vartheta}-1}{\sqrt{\vartheta}+1},\vartheta=\frac{1+\rho_{\texttt{com}}}{1-\rho_{\texttt{com}}}$. Thus, the minimum value of $K$ that leads to $\rho_C\leq\rho^2_{\texttt{opt}}$ can be obtained as $K=\ceil{\log_c\bracket{ 1/\rho^2_{\texttt{opt}}+\sqrt{1/\rho^4_{\texttt{opt}}-1}   }}$. Note that to be used, $A$ must be returned as nonsingular. More details of Chebyshev acceleration applied to the $ABC$-Algorithm along with some numerical results can be found in \cite{XuSunScutariJ,XUCAMSAP_arxiv}.

\subsubsection{Improvement upon existing results and tuning recommendations}\label{subsec_comparison} 
Theorems~\ref{thm:contraction_T_c_T_f_T_B} and \ref{thm:contraction_T_c_T_f}    improve upon existing convergence conditions and rate bounds. A comparison with   notable  distributed algorithms in the literature is presented in Table~\ref{tab:optim_linear_rate}. Since all the schemes therein are special cases of  Algorithm \eqref{alg:ABC_no_G} [with the exception of \cite{alghunaim2019linearly}  that is an instance of Algorithm \eqref{alg:g-ABC}]  (cf. Table \ref{tab:instances})  and satisfy  Assumption \ref{assum:conditions_convergence_no_reg} (or Assumption \ref{assum:conditions_convergence}), one can readily apply Theorem \ref{thm:contraction_T_c_T_f_T_B} (or Theorem \ref{thm:contraction_T_c_T_f}) and determine, for each of them,  a new  stepsize range  and achievable rate: the column ``Stepsize/this paper (optimal, Corollary \ref{col:optimal_D})'' reports   the stepsize value $\gamma^\star(D)$ for the different algorithms  (i.e., given  $B$, $C$ and $D$) while the column ``Rate/$\delta$ this paper'' shows the resulting provably  rate, as given in \eqref{eq:def_lambda_simp_2}.   A direct comparison with the columns ``Stepsize/literature (upper bound)'' and ``Rate/$\delta$, literature'' respectively,  shows that our theorems provide strictly larger ranges for the stepsize of EXTRA~\cite{shi2015extra}   NEXT~\cite{di2016next}/AugDGM~\cite{xu2015augmented,xu2017convergence}   
and Exact Diffusion~\cite{yuan2018exact_p1}, and faster linear rates for {\it all} the algorithms in the table. 

Furthermore, since the rates in the column ``Rate/$\delta$ this paper'' are obtained for the optimal  stepsize value (in the sense of Corollary \ref{col:optimal_D}) of the associated algorithm,  Table~\ref{tab:optim_linear_rate}  also serves as comparison of the convergence rates {\it provably achievable} by the different algorithms. For instance, we notice that, 
 although EXTRA and NIDS both require one communication per gradient evaluation,   NIDS is provably faster, achieving a linear rate of $\delta^\star \log (1/\epsilon)$, with $\delta^\star$ defined in (\ref{rate-separation}), versus the linear rate $(\kappa/(1-\rho)) \log (1/\epsilon)$ of EXTRA. In Sec. \ref{sec:sc_simul} we show that the ranking based on our theoretical findings in  Table~\ref{tab:optim_linear_rate} is   reflected by our numerical experiments--see Fig. \ref{fig_sc_pd}

 \subsubsection{Generalizing existing algorithms to the case $G\neq 0$}
  All the algorithms listed in   Table I but \cite{di2016next} and \cite{alghunaim2019linearly}  are designed for Problem \eqref{prob:dop_nonsmooth_same}  with $G= 0$. Since they are special cases of our general framework and Algorithm (\ref{alg:g-ABC}) can deal with the case  $G\neq 0$, they  inherit the same feature. Their ``proximal'' extension is given by \eqref{eq:g-pd-ATC_eliminate_y}, with the matrices $A$, $B$, and $C$ as in original algorithm (cf. Table II).  Theorem~\ref{thm:contraction_T_c_T_f} and Corollary~\ref{col:contraction_overall_optimal_choice} show that these new algorithms enjoy the same convergence rates of their ``no-proximal'' counterpart.    
For instance, consider  {AugDGM},  corresponding to Algorithm \eqref{alg:ABC_no_G} with $A=B=W^2,\,D=I,\,C=(I-W)^2$; it   clearly satisfies  Assumption~\ref{assum:conditions_convergence} for $W\succ 0$. Its extension to the general optimization with $G\neq 0$ comes readily  substituting these choices of $A,B,C$ into \eqref{eq:g-pd-ATC_eliminate_y} (or Algorithm \ref{alg:ABC_no_G}), yielding  
\begin{equation}\label{eq:g-gradtrack-ATC_eliminate_y_prox}
\begin{aligned}
& \x^{k+1}=\prox{\gamma g}{\z^{k+1}}, \\
& \z^{k+2}=(2\W-\W^2)\z^{k+1}+W^2(\x^{k+1}-\x^k)\nonumber\\
&\quad\quad\quad  -\gamma \,W^2(\nabla f(\x^{k+1})-\nabla f(\x^k)).
\end{aligned}
\end{equation}
As second example, consider   the primal-dual scheme such as NIDS and Exact Diffusion; they correspond to   Algorithm \eqref{alg:ABC_no_G} with  $A=B=\frac{I+W}{2},C=\frac{I-W}{2}$.  Similarly,  we can introduce  their ``proximal'' version as follows:
\begin{equation}\label{eq:g-gradtrack-ATC_eliminate_y_prox}
\begin{aligned}
 \x^k&=\prox{\gamma g}{\z^k}, \\
 \z^{k+2}&=\frac{\I+\W}{2}\left(\z^{k+1}+\x^{k+1}-\x^k\nonumber\right.\\
&\quad \left. -\gamma(\nabla f(\x^{k+1})-\nabla f(\x^k)\right).
\end{aligned}
\end{equation}

\subsection{Application to statistical learning} 
We customize our rate results to the instance of \eqref{prob:dop_nonsmooth_same} modeling    statistical learning tasks over networks. This is an example where the local strong convexity and smoothness constants of the agent functions are different; sill, we will show that, when the data  sets across the agents are sufficiently similar, the rate achieved by the proposed algorithm is within $\widetilde{\mathcal{O}}(1/\sqrt{n})$ the rate of the centralized gradient algorithm.

Suppose each agent $i$ has access to $n$ $i.i.d.$ samples $\{z_j\}_{j \in \mathcal{D}_i}$ following  the distribution $\mathcal{P}.$  
The goal is to learn a model parameter $x$ using  the samples from all the agents; mathematically, we aim at solving the following empirical risk minimization problem:  $\min_{x\in \mathbb{R}^d} \sum_{i\in[m]} \sum_{j\in \mathcal{D}_i} \ell(x;z_j),$ where $\ell(x;z_j)$ is the  loss function measuring the fitness of the statistical model parameterized by $x$ to sample $z_j$; we  assume each $\ell(x;z_j)$ to be quadratic  in $x$ and satisfy $\widetilde{\mu}I \preceq \nabla^2 \ell(x;z) \preceq \widetilde{L}I$, for all $z$. 
This problem is an instance of Problem~\eqref{prob:dop_nonsmooth_same} with $f_i(x) \triangleq \sum_{j\in \mathcal{D}_i} \ell(x;z_j)$.
 Denote the largest and the smallest eigenvalues of $\nabla^2 f_i(x)$ (resp. $\nabla^2 F(x)$) as $L_i$ and $\mu_i$ (resp. $\bar{L}$  and $\bar{\mu}$).
  Then,  each $f_i(x)$ is $\mu \triangleq \min_{i\in [m]}\mu_i$-strongly convex and $L \triangleq \max_{i\in [m]} L_i$-smooth.  Recalling $\kappa = L/\mu,$ the rate in \eqref{rate-separation} reduces to $\left( ({\kappa-1})/({\kappa+1}) \right)^2$, when   $1-\lambda_2(C) \leq \left( ({\kappa-1})/({\kappa+1}) \right)^2$ (possibly using multiple rounds of communications), resulting in      $\mathcal{O}\left( \kappa \log\left({1}/{\epsilon}\right)\right)$ overall number of gradient evaluations.  
  On the other hand, the complexity of the 
 centralized  gradient descent algorithm reads $\mathcal{O}\left( \frac{\bar{L}}{\bar{\mu}}\log\left(\frac{1}{\epsilon}\right)\right)$. To compare these two quantities, compute
\begin{align*}
& \abs{\frac{L}{\mu} - \frac{\bar{L}}{\bar{\mu}} }= \frac{\abs{L \bar{\mu} - \bar{L}\mu}}{\mu \bar{\mu}} \leq \frac{\abs{L- \bar{L}}\bar{\mu} + \bar{L}\abs{\bar{\mu}- \mu}}{\widetilde{\mu}^2} \\
&  \leq \frac{1}{\widetilde{\mu}^2} \left(\bar{\mu} \max_{i\in[m]}\abs{L_i - \bar{L} }+ \bar{L}\max_{i\in[m]} \abs{\mu_i - \bar{\mu}} \right) \\
& \stackrel{(a)}{\leq}  \frac{\bar{\mu} + \bar{L}}{\widetilde{\mu}^2}  \sqrt{\frac{32\widetilde{L}^2 \log (dm/\delta) }{n},  } \quad \text{with probability $1-\delta$}\\
&  \leq 8\sqrt{2}\, \frac{\widetilde{L}^2}{\widetilde{\mu}^2}  \sqrt{\frac{ \log (dm/\delta) }{n},  }
\end{align*}
where in (a) we used the following two facts:  \cite[Corollary~6.3.8]{horn2012matrix} \begin{align}\label{eq:pert_theo}
\max_{i\in[m]} \left(\abs{\mu_i-\bar{\mu}},~\abs{L_i- \bar{L}} \right) \leq  \norm{\nabla^2 f_i(x) - \nabla^2 f(x)},
\end{align}
and \cite[Lemma~2]{shamir2014communication}
\begin{align}\label{eq:concentra}
\max_{i\in[m]} \norm{\nabla^2 f_i(x) - \nabla^2 f(x)}\leq \sqrt{\frac{32\widetilde{L}^2 \log (dm/\delta) }{n}  }
\end{align}with probability at least $1-\delta$.
Therefore, the complexity of our algorithm becomes $\mathcal{O}\left( \left(\frac{\bar{L}}{\bar{\mu}}+ \widetilde{\mathcal{O}} \left( \frac{\bar{L}^2}{\bar{\mu}^2} \frac{1}{\sqrt{n}} \right)\right)\cdot\log\left(\frac{1}{\epsilon}\right)\right)$, with $\widetilde{\mathcal{O}}$ hiding the factor $\log(dm/\delta)$. This shows that when agents have enough data locally ($n$ is large), the above rate is of the same order of that of the centralized  gradient descent algorithm.

\section{Sublinear Convergence (convex case)}\label{sec:sublinear}
We consider now Problem~\eqref{prob:dop_nonsmooth_same} when $f_i$'s are assumed to be convex ($\mu=0$) but not strongly-convex.
We study the sublinear convergence for two splitting schemes, namely: i)   $T=T_CT_fT_B$ applied to  \eqref{prob:dop_nonsmooth_same} with $G=0$;  and ii)  $T=T_CT_gT_fT_B$ applied to \eqref{prob:dop_nonsmooth_same} with $G\neq 0$.\vspace{-0.3cm}
\subsection{Convergence under $G=0$}
We establish  sublinear convergence of  Algorithm~\eqref{alg:ABC_no_G} (corresponding to $T=T_CT_fT_B$)  under the following assumption. 
\begin{assum}\label{assum:sublinear} 
The weight matrices $A\in \mathbb{R}^{m\times m}$, $B,\,C\in \mathbb{S}^{m}$ satisfy: 
\begin{itemize} 
  \item[i)] $A=BD$, with $B \succeq 0$ and $0\prec D\in \mathbb{S}^{m}$; 
  \item[ii)] $ D \ones = \ones$ and $\ones^\top B = \ones^\top$; 
  \item[iii)]  $C \succeq 0$ and  $\Null{{C}}=\Span{\ones}$;
  \item[iv)] $B$ and $C$ commute;
  \item[v)]    $I-\frac{1}{2}C- \sqrt{B}D\sqrt{B} \succeq 0$ \newline ($\Leftrightarrow I-\frac{1}{2}C- A \succeq 0$, if $B$ commutes with $D$).
\end{itemize}

\end{assum}

We quantify the progress of algorithms towards optimality in this setting using the following merit function:
\[
M(\x)\triangleq\max \left\{\norm{(I-J)\x}\norm{\nabla f(X^\star)},\abs{f(\x)-f(X^\star)}\right\},
\]
where $X^\star \triangleq \ones (x^\star)^\top$. Note that the first term encodes consensus errors while the second term measures the optimality gap.

We begin by rewriting Algorithm~\eqref{alg:ABC_no_G} in an equivalent form given in Lemma~\ref{lem:ABC_equiv_form}, which does not have a mixing matrix multiplied to the gradient term.
\begin{lem}~\label{lem:ABC_equiv_form}
	Suppose Assumption~\ref{assum:B_C_commute} holds. Then, Algorithm~\eqref{alg:ABC_no_G} can be rewritten as
	\begin{subequations}\label{eq:ABC_alg_equi} 
		\begin{align}
		& X^k = B\underline{X}^k, \label{eq:ABC_alg_equi_X}\\
		& \underline{X}^{k+1}= D X^k-\gamma (\nabla f(X^k)+\underline{Y}^k), \label{eq:ABC_alg_equi_z}\\
		& \underline{Y}^{k+1}=\underline{Y}^k+\frac{1}{\gamma}C\underline{X}^{k+1} .\label{eq:ABC_alg_equi_Y}
		\end{align}
	\end{subequations}
\end{lem}
\begin{proof}
	since $Y^0 = 0$, we know $\Span{X^1} , \Span{Y^1}\subset \Span{B}$.  It is easy then to deduce from induction that $\Span{X^k}, \, \Span{Y^k} \subset \Span{B},\, \forall k.$  Setting $Y^k=\gamma B\underline{Y}^k$ and $X^k=B\underline{X}^k$ leads to this equivalent form.
\end{proof}

Define $\phi(X,Y) = f(X)+\innprod{Y}{X}.$  In  Lemma~\ref{lem:basic_inequality_I_sublinear} and~\ref{lem:basic_inequality_II_sublinear} below, we  establish  two fundamental inequalities on $\phi(X^k,Y)$ and $\phi(X,Y)$ for $X\in \Span{\ones}$ and $Y\in\Span{C}$, instrumental to prove the sublinear rate. The proofs can be found in the Appendix.

\begin{lem}\label{lem:basic_inequality_I_sublinear}
	Consider the setting of Theorem~\ref{thm:unified_alg_sublinear_no_reg}, let $\{ X^k, \underline{X}^k, \underline{Y}^k\}_{k \in \mathbb{N}_+}$ be the sequence generated by Algorithm~\eqref{eq:ABC_alg_equi} under Assumption~\ref{assum:sublinear}.   Then,  it holds: 
	\begin{equation}\label{eq:basic_inequality_I_sublinear}
	\begin{aligned}
	&\phi(X^{k+1},Y) &\\
	&\leq \phi(X,Y)-\frac{1}{\gamma}\norm{\underline{X}^{k+1}}_{B-BC-AB}^2 \\
	& ~~~ -\frac{1}{\gamma}\innprod{X^{k+1}-X^k}{X^{k+1}-X}_D\\
	&~~~-\gamma\innprod{\underline{Y}^{k+1}-Y}{\underline{Y}^{k+1}-\underline{Y}^k}_{B'}+\frac{L}{2}\norm{X^{k+1}-X^k}^2,
	\end{aligned}
	\end{equation}
	for all $ \,X\in\Span{\ones}$ and $ Y\in\Span{C}$, where $B'=(C+bJ)^{-1}B$,   $b\geq 2$.
\end{lem}

\begin{lem}\label{lem:basic_inequality_II_sublinear}
	Under the same conditions as in Lemma~\ref{lem:basic_inequality_I_sublinear}, if $\gamma\leq \frac{\lambda_{\min}(D)}{L}$, then 
	\begin{equation}\label{eq:basic_inequality_II_sublinear}
	\begin{aligned}
	& \phi(\widehat{X}^{k},Y)-\phi(X,Y)\\
	& \leq \frac{1}{2k}\Big(\frac{1}{\gamma}\norm{X^0-X}^2_D+\gamma\frac{\rho(B-J)}{\lambda_2(C)}\norm{Y}^2\Big),
	\end{aligned}
	\end{equation}
	for all $ X\in\Span{\ones}$ and $ Y\in\Span{C}$, where $\widehat{X}^{k} \triangleq \frac{1}{k} \sum_{t=1}^k X^t$.
\end{lem}

We  now  prove the sublinear convergence rate. 
\begin{thm}[Sublinear rate for $T_CT_fT_B$]
	\label{thm:unified_alg_sublinear_no_reg}
	Consider Problem~\eqref{prob:dop_nonsmooth_same} under Assumption~\ref{assum:smooth_strong_reg} with $\mu=0$ and $G=0$; and let $x^\star$ be an optimal solution.  Let $\{(\x^k,Y^k)\}_{k\in \mathbb{N}_+}$ be the sequence generated by Algorithm~\eqref{alg:ABC_no_G} under Assumptions~\ref{assum:sublinear}. Then, if $0<\gamma \leq \frac{\lambda_{\min}(D)}{L}$, we have   
	\begin{equation}\label{eq:sublinear_rate_nonsepa} 
	\begin{aligned}
	&M(\widehat{\x}^{k})  \leq \frac{1}{{k}}\Big(\frac{1}{2\gamma}\norm{\x^0-\x^\star}^2_D+ 2 \gamma \,\frac{\rho(B-J)}{ \lambda_2(C)}\norm{\nabla f(\x^\star) }^2 \Big),
	\end{aligned}
	\end{equation}
	where $\widehat{X}^{k}=\frac{1}{{k}}\sum_{t=1}^{{k}}X^t$.
\end{thm}
\begin{proof}
Setting $X = X^\star$ in  \eqref{eq:basic_inequality_I_sublinear},  it holds
\[
\begin{aligned}
& \phi(\widehat{X}^{k},Y)-\phi(X^\star,Y) = f(\widehat{X}^{k}) - f(X^\star) - \innprod{\widehat{X}^{k} - X^\star}{Y}\\ 
& =   f(\widehat{X}^{k}) - f(X^\star) - \innprod{\widehat{X}^{k} }{Y} \leq h(\norm{Y}),
\end{aligned}
\]
for $Y \in \Span{C}$, where $h(\cdot)=\frac{1}{2k}\Big(\frac{1}{\gamma}\norm{X^0-X^\star}^2_D+\gamma\frac{\rho(B-J)}{\lambda_2(C)}(\cdot)^2\Big)$.
Setting $Y=-2\frac{(I-J)\widehat{X}^k}{\norm{(I-J)\widehat{X}^k}}\norm{Y^\star }$, with $Y^\star = -\nabla f(X^\star)$, we have
\[
f(\widehat{X}^k)-f(X^\star)+2\norm{Y^\star}\norm{(I-J)\widehat{X}^k}\leq h(2\norm{Y^\star}).
\]
By the convexity of $f$, $f(\widehat{X}^k)-f(X^\star)+\innprod{(I-J)\widehat{X}^k}{Y^\star}=f(\widehat{X}^k)-f(X^\star)+\innprod{\widehat{X}^k}{Y^\star}\geq 0$, we have $f(\widehat{X}^k)-f(X^\star)\geq -\norm{Y^\star}\norm{(I-J)\widehat{X}^k}$.  Combining the above two relations, we have 
$M(\widehat{X}^k)\leq h(2\norm{Y^\star}).$ This completes the proof.
\end{proof}

Finally, we provide the  choice of $\gamma$ that optimizes the rate given  in Theorem~\ref{thm:unified_alg_sublinear_no_reg}.

\begin{col}\label{col:unified_alg_sublinear_no_reg}
Consider  the setting of Theorem~\ref{thm:unified_alg_sublinear_no_reg}. The stepsize that minimizes the right hand side of~\eqref{eq:sublinear_rate_nonsepa} is  
\begin{equation}\label{eq:ss_sublinear}
\gamma =  \min\left(\frac{\lambda_{\min}(D)}{L},\, \frac{1}{2}\sqrt{\frac{\lambda_2(C)}{\rho(B-J)}} \frac{\norm{\x^0-\x^\star}_D}{\norm{\nabla f(\x^\star)}}\right),
\end{equation}
leading to a sublinear rate
\begin{equation}\label{eq:sublinear_rate}
\begin{aligned}
&M(\widehat{\x}^k) \leq \frac{1}{k}\max\bigg\{\frac{L  \norm{\x^0-\x^\star}_D^2}{\lambda_{\min}(D)},\\
&~~~~~~~~~~~~~2\sqrt{\frac{\rho(B-J)}{\lambda_2(C)}}\norm{\x^0-\x^\star}_D\norm{\nabla f(\x^\star)} \bigg\}.
\end{aligned}
\end{equation}
\end{col}
Note that the stepsize in~\eqref{eq:ss_sublinear} depends on  $\norm{\x^0-\x^\star}_D/\norm{\nabla f(\x^\star)}$, an information that is not generally available;  we discuss this issue in  Sec. ~\ref{sec:discussion_sublinear}. \vspace{-0.2cm}

\subsection{Convergence under $G\neq 0$}
We  consider now Problem~\eqref{prob:dop_nonsmooth_same} with $G\neq 0$   and  $\mu=0$. We study convergence of   a variation of the general scheme \eqref{alg:g-ABC}, where the proximal operator is employed before $T_f$,  yielding the operator decoposiiton $T_CT_gT_fT_B$.\footnote{It is not difficult to check that any fixed point of $T_CT_gT_fT_B$ has the same fixed-points of the operator in \eqref{eq:T_operator}.} 
This scheme reads 
\begin{equation}
\label{alg:ABC_sublinear_prox}
\begin{aligned}
\underline{X}^{k+1}&= DX^k-\gamma ( \nabla f(X^k)+\underline{Y}^k), \\
 X^{k+1} & = \prox{\gamma g}{B\underline{X}^{k+1}},\\
\underline{Y}^{k+1}&=\underline{Y}^k+\frac{1}{\gamma} CX^{k+1}.
\end{aligned}
\end{equation}
Note that a key difference between \eqref{alg:g-ABC} and the above algorithm is that the former uses in the update of the dual variable $Y$ the variable  $Z$, the variable  before the operator $\prox{\gamma g}{\cdot}$,  while the latter uses the variable $X$, i.e., the variable after the operator $\prox{\gamma g}{\cdot}$.
It is not difficult to check that \eqref{alg:ABC_sublinear_prox}  subsumes many existing proximal-gradient methods, \jm{such as PG-EXTRA~\cite{shi2015proximal} or ID-FBBS~\cite{xu2018bregman} (with $A=\W,B=\I,C=\I-\W$)}.  We present a unified result of the sublinear convergence for the algorithm~\eqref{alg:ABC_sublinear_prox}, under the following assumption.


\begin{assum}\label{assum:sublinear_prox} 
The weight matrices $B,\,C,\,D\in \mathbb{S}^{m}$ satisfy: 
\begin{itemize} 
  \item[i)] $B=I$; 
  \item[ii)] $\ones^\top D \ones = m$; 
  \item[iii)]  $C \succeq 0$ and  $\Null{{C}}=\Span{\ones}$;
  \item[iv)] $0\prec D \preceq \I-\frac{C}{2}$.
\end{itemize}

\end{assum}

Note that the above assumption is,  indeed, a customization of Assumption~\ref{assum:sublinear}. The condition $B=I$ is introduced to deal with   the complication of the proximal operator $T_g$.  

 {We study convergence of Algorithm \eqref{alg:ABC_sublinear_prox} using the following merit function measuring the progresses of the algorithms from consensus and optimality. Define
\[M(\x)\triangleq
\max \left\{\norm{(I-J)\x}\norm{Y^\star},\,\abs{(f+g)(\x)-(f+g)(X^\star)}\right\}.
\]
where $Y^\star = -\left( \nabla f(X^\star) + \ones (\xi^\star)^\top \right)$, for some $\xi^\star \in \partial G(x^\star)$ such that $\xi^\star + \nabla F(x^\star)=\xi^\star + \frac{1}{m} \sum_{i=1}^m \nabla f(x^\star) =0$. Note that, since $\ones^\top Y^\star = 0$, we have $Y^\star \in \Span{C}.$ }

We are now ready to state   our convergence  result, whose proof is left to the Appendix due to its similarity to that of Theorem~\ref{thm:unified_alg_sublinear_no_reg}.

\begin{thm}[Sublinear rate for $T=T_CT_gT_fT_B$]
\label{thm:sublinear_prox} 
Consider Problem~\eqref{prob:dop_nonsmooth_same} under Assumption~\ref{assum:smooth_strong_reg} with $\mu=0$; and let $x^\star$ be an optimal solution.  Let $\{(\x^k,Y^k)\}_{k\geq 0}$ be the sequence generated by Algorithm~\eqref{alg:ABC_sublinear_prox} under Assumptions~\ref{assum:sublinear_prox}.
Then, if $\gamma < \frac{\lambda_{\min}(D)}{L}$,  we have 
\begin{equation}\label{eq:prox_sublinear}
\begin{aligned}
M(\widehat{\x}^k)  \leq \frac{1}{k}\Big(\frac{1}{2\gamma}\norm{\x^0-\x^\star}^2_D+ 2 \gamma \,\frac{1}{ \lambda_2(C)}\norm{\nabla f(\x^\star) }^2 \Big),
\end{aligned}
\end{equation}
where $\widehat{X}^{k}=\frac{1}{{k}}\sum_{t=1}^{{k}}X^t$. 
\end{thm}

\begin{col}\label{col:unified_alg_sublinear_reg}
Consider the setting of Theorem~\ref{thm:sublinear_prox}. The stepsize that minimizes the right hand side of~\eqref{eq:prox_sublinear} is
\begin{equation}\label{eq:ss_sublinear_prox}
\gamma =  \min\left(\frac{\lambda_{\min}(D)}{L},\, \frac{1}{2} \sqrt{\lambda_2(C)} \frac{\norm{\x^0-\x^\star}_D}{\norm{\nabla f(\x^\star)}}\right),
\end{equation}
leading to a sublinear rate
\begin{equation}\label{eq:sublinear_G}
\begin{aligned}
&M(\widehat{\x}^k) \leq \frac{1}{k}\max\bigg\{\frac{L  \norm{\x^0-\x^\star}_D^2}{\lambda_{\min}(D)},\\
&~~~~~~~~~~~~~2\sqrt{\frac{1}{\lambda_2(C)}}\norm{\x^0-\x^\star}_D\norm{\nabla f(\x^\star)} \bigg\}.
\end{aligned}
\end{equation}\vspace{-0.4cm}
\end{col}

\subsection{Discussion}
\label{sec:discussion_sublinear}
\subsubsection{On rate seperation}

Differently from most of the existing works, such as~\cite{shi2015extra,qu2017harnessing,aybat2017distributed}, the above convergence results (Corollary~\ref{col:unified_alg_sublinear_no_reg} and \ref{col:unified_alg_sublinear_reg}) establish the explicit dependency of the rate on the network parameter as well as the properties of the cost functions.  Specifically, the rate coefficients  in \eqref{eq:sublinear_rate} and \eqref{eq:sublinear_G} show an explicit dependence on the network and optimization parameters, with the first term on the RHS  corresponding to the rate of the centralized optimization algorithm while the second term related to both the communication network and the heterogeneity of the cost functions of the agents (i.e., $\norm{\nabla f(x^\star)}$).  The smaller $\norm{\nabla f(x^\star)}$,  the more similar the objective functions agents have. For instance, when $f_i$'s share a common minimizer, i.e., $\norm{\nabla f(x^\star)}=0$, the  rate will reduce to the centralized one.   The term $\sqrt{\frac{\rho(B-J)}{\lambda_2(C)}}$  accounts for the network effect on the rate. For instance, set $C=I-B$, so that 
  $\lambda_2(C) = 1-\rho(B-J).$ If $\rho(B-J)\to0$ (meaning a network tending to a fully connected graph), $\sqrt{\frac{\rho(B-J)}{\lambda_2(C)}}\to 0$, leading to the rate of the centralized gradient algorithm [cf. (\ref{eq:sublinear_rate})].  On the other hand, if $\rho(B-J)\to 1$ (poorly connected network),
 $\sqrt{\frac{\rho(B-J)}{\lambda_2(C)}}\to +\infty$, deteriorating the overall rate. 
  As a result,   when the agents have similar cost functions (i.e., small value of $\norm{\nabla f(x^\star)}$) or the network is well connected, 
  the first term will dominate the second, leading to the centralized performance.  The tightness of the rate expression (Corollary~\ref{col:unified_alg_sublinear_no_reg}) is validated by our numerical results--see Sec.~\ref{sec:sim_smooth}.

\subsubsection{On the choice of stepsize}

The optimal stepsize, as indicated in~\eqref{eq:ss_sublinear} (resp. \eqref{eq:ss_sublinear_prox}), is such that the two terms in \eqref{eq:sublinear_rate_nonsepa} (resp. \eqref{eq:prox_sublinear}) are balanced.
Albeit \eqref{eq:ss_sublinear} and \eqref{eq:ss_sublinear_prox} generally are not implementable, due to the unknown quantity $\norm{\x^0-\x^\star}_D/\norm{\nabla f(\x^\star)}$, the result is interesting on the theoretical side, showing that  the ``optimal'' stepsize is not necessarily $1/L$ but depends on the the network and the degree of heterogeneity of the cost functions as well. In particular, the optimal choice is $1/L$ when the network is well connected and agents share similar ``interests'', i.e., $\norm{\nabla f(x^\star)}$ is small. On the other hand,   as the connectivity of the network becomes worse and/or the heterogeneity of local cost functions becomes larger,  stepsize values  smaller  than $1/L$  ensure better performance. 
This observation provides recommendations on stepsize tuning and it is validated by our numerical experiments.

\section{Numerical Results}\label{sec:numerical}

We present some numerical results on strongly convex and convex instances of \eqref{prob:dop_nonsmooth_same}, supporting  our theoretical findings. 
The obtained rates bounds  are shown to predict well the practical behavior of the algorithms. For instance, the ATC-based schemes exhibit a clear rate separation [as predicted by (\ref{rate-separation})]: the convergence rate  cannot be continuously improved  by unilaterally decreasing the difficulty of the problem or increasing the connectivities of the communication matrices. \vspace{-0.2cm}

\subsection{Strongly convex problems}\label{sec:sc_simul}
We consider a regularized least squares problem over an undirected graph consisting of $50$ nodes, generated through the Erdos-Renyi model with activating probability of $0.05$ for each edge. The problem reads 
\begin{align}\label{p:elastic_net}
\min_{x\in \mathbb{R}^d} \left(\frac{1}{50} \sum_{i=1}^{50}\norm{U_i x - v_i}^2\right) + \rho \norm{x}_2^2 + \lambda \norm{x}_1.
\end{align}
where $U_i\in \mathbb{R}^{r\times d}$ and $\v_i\in \mathbb{R}^{r\times 1}$ are the feature vector and labels, respectively, only accessible by node $i$. 
For brevity, we denote $U = [U_1; U_2;\cdots; U_{50}] \in \mathbb{R}^{50r \times d}$ and $v = [v_1; v_2;\cdots; v_{50}] \in \mathbb{R}^{50r \times 1}$ and use ${M}_{:,i}$ (resp. ${M}_{i,:}$)  to denote the $i$-th column (resp. row) of a matrix ${M}$.  In the simulation, we set $r = 20$, $d = 40$, $\rho=20$ and $\lambda = 1.$   We generate the matrix $U$ of the feature vectors according to the following  procedure, proposed in \cite{agarwal2010fast}:  we first generate an innovation matrix $Z$ with each entry  i.i.d. drawn from $\mathcal{N}(0,1).$  Using a control parameter $\omega \in [0,1)$, we then generate columns of $U$  such that the first column is  $U_{:,1} = {Z}_{:,1}/\sqrt{1 - \omega^2}$ and the rest are recursively set as $U_{:,i} = \omega U_{:,i-1} + Z_{:,i}$, for $i = 2,\ldots, d$. As a result, each row $U_{i,:}\in \mathbb{R}^d$ is a Gaussian random vector and its covariance matrix $\Sigma = \text{cov}(U_{:,i})$ is the identity matrix if $\omega = 0$ and becomes extremely ill-conditioned as $\omega \to 1.$  Finally, we generate $x_0 \in \mathbb{R}^d$ with sparsity level $0.3$ and each nonzero entry   i.i.d. drawn  from $\mathcal{N}(0,1)$, and set $v= Ux_0 + \xi,$ where each component of the noise $\xi$ is  i.i.d. drawn from $\mathcal{N}(0,0.04).$  By changing $\omega$ one can control the conditional number  $\kappa$ of the smooth objective in (\ref{p:elastic_net}). \\\  
$\bullet$ \textbf{Validating \eqref{rate-separation}:} We simulated the following instances of Algorithm~\ref{alg:g-ABC}. We set  $A = B = (\frac{I+W}{2})^K$ and $C=I-B$, where $W$ is a weight matrix generated using the Metropolis-Hastings rule~\cite{xiao2004fast}, and  $K\geq 1$ is the number of inner consensus steps.  When Chebyshev acceleration is employed in the inner consensus steps, we instead used  $A=B=({I+P_K(\widetilde{W})})/{2}$ and $C=I-B$ (condition of Corollary~\ref{col:contraction_overall_optimal_choice} is satisfied).  In Figure~\ref{fig_sc_pd}, we plot the number of iterations (gradient evaluations) needed by the algorithm to reach an accuracy of $10^{-8}$, versus the number of inner consensus $K$, for different values of $\kappa$; solid (resp. dashed) line-curves refer to non-accelerated (Chebyshev) consensus steps. The markers (diamond symbol) correspond to the number of iterations predicted by  \eqref{rate-separation} for the max in  \eqref{rate-separation} to achieve the minimum value, that is,   $\ceil{{2\log (\frac{\kappa-1}{\kappa+1})}/{\log (\frac{1+\lambda_{m-1}(W)}{2})}}$. The following comments are in order. \textbf{(i)} As $K$ increases, the number of iterations needed to reach the desired solution accuracy decreases till it reaches a plateau; further communication rounds do not improve the performance, as the optimization component becomes the bottleneck [as predicted by \eqref{rate-separation}]. \textbf{(ii)} Less number of iterations are needed when  $\kappa$ becomes smaller (simpler problem). Finally,  \textbf{(iii)} Chebyshev  acceleration further reduces the number of iterations. This was all predicted by our theoretical findings.



\begin{figure}[t]
\centering
\includegraphics[width=0.42\textwidth]{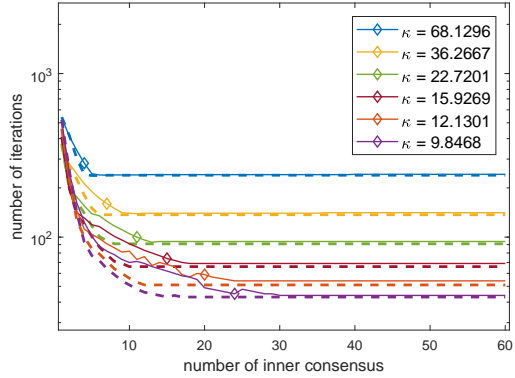}\vspace{-0.2cm}
\caption{ LASSO problem: Number of iterations (gradient evaluations) needed to reach an accuracy of $10^{-8}$ by Algorithm~\ref{alg:g-ABC} employing Chebyshev acceleration (dashed lines) and multiple rounds of consensus (solid lines).}
\label{fig_sc_pd}\vspace{-0.3cm}
\end{figure}

\noindent$\bullet$ \textbf{Validating Table~\ref{tab:optim_linear_rate}: Comparison of the ``prox''-versions of existing algorithms. }
In Fig.~\ref{fig:ranking_good_graph} we compare the ``prox'' version of several existing algorithms, applied to  \eqref{p:elastic_net}: we plot the optimality gap $\norm{\x^k-\ones x^{\star\top}}$ versus the overall number of iterations (gradient evaluations). The setting   is the same as in the previous example, except that now we set $\omega = 0.8.$  The stepsize of each algorithm is chosen according to \eqref{eq:opt_ss_D}. The network is simulated according to the Erdos-Renyi model with a connection probability of $0.25$; in this setting, the max in  \eqref{rate-separation} is achieved at $(\kappa-1)/(\kappa+1)$. It follows from the figure that ATC-based schemes, such as Prox-NEXT/AugDGM, Prox-NIDS, outperform non-ATC ones, such as Prox-EXTRA and Prox-DIGing, validating the ranking established in (the last column of) Table~\ref{tab:optim_linear_rate}.\vspace{-0.3cm}

\begin{figure}[t]
\centering
\includegraphics[width=0.38\textwidth]{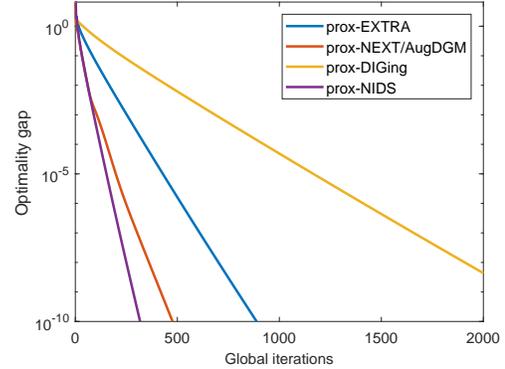}\vspace{-0.3cm}
\caption{Performance comparison of the proximal extensions of some existing algorithms--the schemes are all new and instantiations of (\ref{alg:g-ABC}) with the weight-matrices as in Table II for their counterparts applied to \eqref{prob:dop_nonsmooth_same} with $G=0$. 
}
\label{fig:ranking_good_graph}\vspace{-0.5cm}
\end{figure}

\subsection{Non-strongly-convex problems}
\label{sec:sim_smooth}

To illustrate the results for non-strongly convex problems, we report here a logistic regression problem using the Ionosphere Data Set as follows \cite{Dua:2019}:
\[
\min_{x\in \mathbb{R}^{34}} \frac{1}{50} \sum_{i=1}^{50} \sum_{k=7(i-1)+1}^{7i} \log(1+\exp(-v_k u_k^\top x)),
\]
where $u_k\in \mathbb{R}^{34}$ and $v_k\in \{-1,1\}$ are respectively the feature vector and label of the $k$-th sample.  We use $U = [u_1, u_2, \cdots, u_{350}]^\top$ to denote the feature matrix.  We construct several problems with different Lipschitz constant by multiplying the feature matrix $U$ with different scaling factors. In particular, given the original problem with an $L$-smooth objective function $f$, one can multiply $U$ by a scalar $0<\alpha <1$ to construct a new $\alpha^2 L$-smooth objective function $f_\alpha(\cdot)$.
In the simulation, we consider the polynominal method and thus set $A = B = \widetilde{W}^K$ and $C=I-B$.  The stepsize of the algorithm is chosen\footnote{This choice is not implementable in practice but only for illustration.} according to \eqref{eq:ss_sublinear}. Figure~\ref{fig_sublinear} plots the number of iterations (gradient evaluations) needed by the algorithm to reach an accuracy of $10^{-4}$ in solving different problems with different difficulty versus the number of inner loop of consensus. It follows from the figure that, similar as with the strongly convex case, the number of iterations needed is decreasing with the number of inner loops of consensus, until it reaches to a turning point which appears later as the Lipschitz constant $L$ decreases. This observation verifies the result as shown in
 \eqref{eq:sublinear_rate} where the two quantities is to be properly balanced with multiple communication steps.\vspace{-0.3cm}

\begin{figure}[t]\vspace{-0.4cm}
\centering
\includegraphics[width=0.4\textwidth]{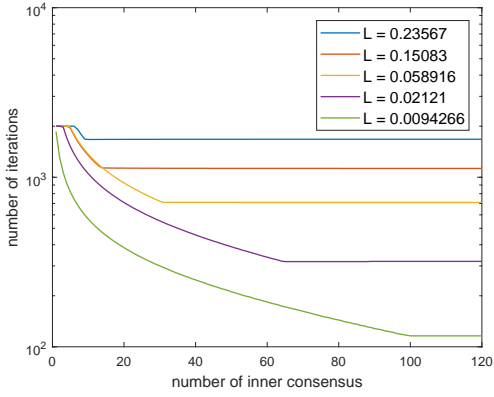}
\caption{Logistic regression problem:  Number of iterations (gradient evaluations) needed to reach an accuracy of $10^{-4}$ by Algorithm~\ref{alg:ABC_no_G} (equivalently Algorithm~\ref{eq:ABC_alg_equi}) employing multiple rounds of consensus.}
\label{fig_sublinear}
\end{figure}

\section{Conclusion} 
We proposed a unified distributed algorithmic framework for composite optimization problems over   networks; the framework subsumes many existing schemes. When the agents' functions are strongly convex, linear convergence is proved      leveraging an operator contraction-based analysis. With a proper choice of the design parameters, the rate dependency on the network and cost functions can be decoupled, which permits  to achieve the rate of the  centralized (proximal)-gradient methods  using a finite number of communications per gradient evaluations. Our convergence conditions and rate bounds   improve on existing ones. Furthermore, thanks to our unified framework and analysis, a fair comparison and ranking of the different   (including existing) schemes were provided. 
 When the functions of the agents are (not strongly) convex, a sublinear convergence rate was established, shading  light on the dependency of the convergence  on the connectivity of the network and the heterogeneity of the cost functions.

\appendix
\subsection{Proof of Lemma~\ref{lem:iff_fixed_point}}\label{sec:pf_lm2}

$(\Leftarrow):$ Suppose Assumption~\ref{assum:cond_A_B} hold. First, for any $\x\in\mathcal{S}_\texttt{Fix}$, we have $\Span{\x} \subset \Null{\C}=\Span{\ones}$ and so $\ones^\top(I-\A)\x=0$.  Then we have $\ones^\top\nabla f(\x)=\ones^\top\B\nabla f(\x)\in - \ones^\top\partial g(\x),$ i.e., $\exists \,  \xi \in \partial g(\x)$ such that $\Span{\nabla f(\x)  + \xi }\perp \Span{\ones} = \Null{\sqrt{\C}},$  which implies that $\Span{\nabla f(\x)  + \xi} \subset \Span{\sqrt{C}}.$  Therefore, $\exists\, \y \in  \mathbb{R}^{m\times d}$ such that $\nabla f(\x)  + \xi = -\sqrt{C} \y,$ i.e., $\nabla f(X)+  \sqrt{{C}}\y \in-\partial g(\x).$  Hence, $\x\in\mathcal{S}_\texttt{KKT}.$  Secondly, for any $\x\in\mathcal{S}_\texttt{KKT}$, we have $\Span{\x} \subset \Span{\ones}$ and so $\ones^\top(I-\A)\x+\gamma \ones^\top\B\nabla f(\x) = \gamma \ones^\top \left( \nabla f(X)+  \sqrt{{C}}\y \right) \in -\gamma \ones^\top\partial g(\x)$, i.e., $\x\in\mathcal{S}_\texttt{Fix}$.\\
$(\Rightarrow:)$ $\mathcal{S}_\texttt{KKT}=\mathcal{S}_\texttt{Fix}$ implies that, for any arbitrarily given $f$, $g$ and $\x$, if $\Span{\x} \subset\Span{\ones}$ and $\ones^\top\nabla f(\x)\in -\ones^\top\partial g(\x)$, it must be $\ones^\top(I-\A)\x+\gamma \ones^\top\B\nabla f(\x)\in-\gamma \ones^\top\partial g(\x)$, which, due to the arbitrary nature  of $f$, $g$, and $\x$, further implies   $\ones^\top(I-\A)\ones=0$ and $\ones^\top\B = \ones^\top.$

\section{Supporting Proofs of Linear Convergence Rate}

\subsection{Proof of Lemma~\ref{lem:contraction_T_f}}\label{sec:pf_tf}

Since $ 0\prec D \preceq I$, we have
\begin{equation}\label{eq:pf_tf_1}
\begin{aligned}
&\norm{DX-\gamma\nabla f(X)-DY+\gamma\nabla f(Y)}^2\\
& \leq \norm{DX-\gamma\nabla f(X)-DY+\gamma\nabla f(Y)}^2_{D^{-1}} \\
&=\norm{X-Y}^2_{D}-2\gamma\innprod{X-Y}{\nabla f(X)-\nabla f(Y)}\\
&~~~+\gamma^2\norm{\nabla f(X)-\nabla f(Y)}^2_{D^{-1}}.
\end{aligned}
\end{equation}
Then we proceed to lower bound $\innprod{X-Y}{\nabla f(X)-\nabla f(Y)}.$
Let $\x' =\sqrt{D}\x, ~\tilde{f}(\x)=f(\sqrt{D^{-1}}\x)$.  Given any two points $X,Y\in\mathbb{R}^{m\times d}$, we have
\[
\begin{aligned}
&\innprod{X-Y}{\nabla f(X)-\nabla f(Y)}\\
&=\innprod{\sqrt{D^{-1}}X'-\sqrt{D^{-1}}Y'}{\nabla f(\sqrt{D^{-1}}X')-\nabla f(\sqrt{D^{-1}}Y')}\\
&=\innprod{X'-Y'}{\nabla \tilde{f}(X')-\nabla \tilde{f}(Y')}\\
&\overset{(*)}{\geq}\frac{L'\mu'}{L'+\mu'}\norm{X'-Y'}^2+\frac{1}{L'+\mu'}\norm{\nabla \tilde{f}(X')-\nabla \tilde{f}(Y')}^2\\
&=\frac{L'\mu'}{L'+\mu'}\norm{X-Y}^2_{D}+\frac{1}{L'+\mu'}\norm{\nabla f(X)-\nabla f(Y)}^2_{D^{-1}}
\end{aligned}
\]
where $(*)$ is due to \cite[Theorem~2.1.12]{nesterov2013introductory}, with $L' = \frac{L}{\lambda_{\min}(D)}$ and $\mu' = \frac{\mu}{\lambda_{\max}(D)}.$ Thus, knowing that $0<\gamma \leq \frac{2\lambda_{\min}(D)}{L+\mu \cdot \eta(D)} =\frac{2}{L'+\mu'}$ and continuing from \eqref{eq:pf_tf_1}, we have
\[
\begin{aligned}
&\norm{DX-\gamma\nabla f(X)-DY+\gamma\nabla f(Y)}^2\\
&\leq \left( 1- 2\gamma \frac{L'\mu'}{L'+\mu'} \right)\norm{X-Y}^2_{D} \\
& ~~~ -(\frac{2\gamma}{L'+\mu'}-\gamma^2)\norm{\nabla f(X)-\nabla f(Y)}^2_{D^{-1}}\\
&\leq \left( 1- 2\gamma \frac{L'\mu'}{L'+\mu'} \right)\norm{X-Y}^2_{D}.
\end{aligned}
\]
In particular, if we set $\gamma=\gamma^\star$, we have $1- 2\gamma^\star \frac{L'\mu'}{L'+\mu'} = \left( \frac{L'-\mu'}{L'+\mu'}\right)^2= \left( \frac{\kappa-\eta(D)}{\kappa+\eta(D)} \right)^2$.

\subsection{Proof of Theorem \ref{thm:contraction_T_c_T_f}}
This proof is similar to that of Theorem~\ref{thm:contraction_T_c_T_f_T_B}, except that in the following chain of inequalities, we need to tackle the additional operator $T_g.$  For $\forall \, X, Y\in\mathbb{R}^{2m\times d}$,  $X_\ell,Y_\ell \in \Span{\sqrt{C}}$,
\begin{align*}
& \norm{T\,X- T\,Y}_{\Lambda_{I-C}}^2 \\
&= \norm{T_c \circ T_f \circ T_g \circ T_B\,(X)- T_c \circ T_f \circ T_g \circ T_B\,(Y)}_{\Lambda_{I-C}}^2 \\
& \stackrel{Lm.~\ref{lem:contraction_T_c}} {=}\norm{T_f \circ T_g \circ T_B\,(X)- T_f \circ T_g \circ T_B\,(Y)}_{V_{I-C}}^2 \\
& \stackrel{Lm.~\ref{lem:contraction_T_f}}{\leq} \norm{ T_g \circ T_B\,(X)- T_g \circ T_B\,(Y)}_{V_{I-C} \Lambda_{{q(D,\gamma)}^2 D}}^2 \\ 
& ~~\, \leq ~~ \norm{ T_g \circ T_B\,(X)- T_g \circ T_B\,(Y)}_{V_{I-C} \Lambda_{{q(D,\gamma)}^2 I}}^2 \\ 
&  \stackrel{Lm.~\ref{lem:contraction_T_g}}{\leq} \norm{ T_B\,(X)- T_B\,(Y)}_{V_{I-C} \Lambda_{{q(D,\gamma)}^2 I} }^2 \\
& \stackrel{Lm.~\ref{lem:contraction_T_B}}{=} \norm{X- Y}_{V_{I-C} \,\Lambda_{{q(D,\gamma)}^2 B^2} }^2 \stackrel{(*)}{\leq} \delta \, \norm{X- Y}_{\Lambda_{I-C}}^2,
\end{align*}
where (*) is due to: i) for all $\,(Z)_u \in \mathbb{R}^{m\times d},$
\begin{align*}
& \|(Z)_u\|^2_{B^2}= \|(I-C)^{\frac{1}{2}}(Z)_u\|^2_{B^2(I-C)^{-1}}  \\
&  \leq  \lambda_{\text{max}} (B^2(I-C)^{-1})\|(I-C)^{\frac{1}{2}}(Z)_u\|^2  \\
& =\lambda_{\text{max}} (B^2(I-C)^{-1})\norm{(Z)_u}^2_{I-C} ; 
\end{align*}
and ii) $X_\ell,Y_\ell \in \Span{\sqrt{C}}$. 

\section{Supporting Proofs of Sublinear Convergence Rate}

\subsection{Proof of Lemma~\ref{lem:basic_inequality_I_sublinear}}\label{sec:pf_ine_I}

Since $f$ is $L$-smooth, we have
\begin{equation}\label{eq:L_smooth_inequalitY}
\begin{aligned}
& f(X^{k+1})\\
& \leq f(X^k)+\innprod{\nabla f(X^k)}{X^{k+1}-X^k}+\frac{L}{2}\norm{X^{k+1}-X^k}^2 \\
&\overset{(a)}{\leq} f(X) + \innprod{\nabla f(X^k)}{X^k-X} + \innprod{\nabla f(X^k)}{X^{k+1}-X^k} \\
& ~~~ +\frac{L}{2}\norm{X^{k+1}-X^k}^2 \\
& = f(X)+\innprod{\nabla f(X^k)}{X^{k+1}-X}+\frac{L}{2}\norm{X^{k+1}-X^k}^2.
\end{aligned}
\end{equation}
where $(a)$ is due to the fact that $f(\x)\geq f(\x^k)+\innprod{\nabla f(\x^k)}{\x-\x^k}$ from the convexity of $f$. 

Then, we relate the gradient term $\nabla f(X^k)$ to other quantities using \eqref{eq:ABC_alg_equi_z}  as follows
\[
\begin{aligned}
& \innprod{\nabla f(X^k)}{X^{k+1}-X} =-\frac{1}{\gamma}\innprod{ \underline{X}^{k+1}}{X^{k+1}-X}\\
&~~~+\frac{1}{\gamma}\innprod{DX^k-\gamma \underline{Y}^{k}}{X^{k+1}-X}\\
& = -\frac{1}{\gamma}\innprod{(I-C)\underline{X}^{k+1}}{X^{k+1}-X}\\
&~~~+\frac{1}{\gamma}\innprod{DX^k-\gamma \underline{Y}^{k+1}}{X^{k+1}-X},
\end{aligned}
\]
where we have used \eqref{eq:ABC_alg_equi_Y} to obtain the last relation.
Now, substituting the above relation into \eqref{eq:L_smooth_inequalitY}, we further have
\begin{align}
\begin{split}
& f(X^{k+1})\leq f(X)-\frac{1}{\gamma}\innprod{(I-C)\underline{X}^{k+1}}{X^{k+1}-X}\\
&~~~~~ +\frac{1}{\gamma}\innprod{D X^k}{X^{k+1}-X}-\innprod{\underline{Y}^{k+1}}{X^{k+1}-X} \\
&~~~~~ +\frac{ L}{2}\norm{X^{k+1}-X^k}^2
\end{split}
\end{align}
Adding $\innprod{Y}{X^{k+1}-X}$, with $X\in\Span{\ones}$ and $Y\in \Span{C}$, to both sides of the above equation and noticing $(C+bJ)^{-1}C = I- J$  yields
\begin{align*}
& \phi(X^{k+1},Y)\leq \phi(X,Y)-\frac{1}{\gamma}\innprod{(I- C)\underline{X}^{k+1}}{B(\underline{X}^{k+1}-X)}\\
&~~~+\frac{1}{\gamma}\innprod{D B\underline{X}^k}{B(\underline{X}^{k+1}-X)} +\frac{L}{2}\norm{X^{k+1}-X^k}^2\\
&~~~-\innprod{(C+2J)^{-1}C(\underline{Y}^{k+1}-Y)}{B(\underline{X}^{k+1}-X)}\\
& = \phi(X,Y)-\frac{1}{\gamma}\innprod{(I- C)\underline{X}^{k+1}}{B(\underline{X}^{k+1}-X)}\\
& ~~~ +\frac{1}{\gamma}\innprod{D B\underline{X}^k}{B(\underline{X}^{k+1}-X)} +\frac{L}{2}\norm{X^{k+1}-X^k}^2 \\
&~~~-\innprod{\underline{Y}^{k+1}-Y}{C\underline{X}^{k+1}}_{B'} \\
& = \phi(X,Y)-\frac{1}{\gamma}\innprod{(I-C-DB)\underline{X}^{k+1}}{B(\underline{X}^{k+1}-X)} \\
& ~~~ +\frac{1}{\gamma}\innprod{D B(\underline{X}^k-\underline{X}^{k+1})}{B(\underline{X}^{k+1}-X)}\\
&~~~-\gamma\innprod{\underline{Y}^{k+1}-Y}{\underline{Y}^{k+1}-\underline{Y}^k}_{B'}+\frac{L}{2}\norm{X^{k+1}-X^k}^2,
\end{align*}
where we have used \eqref{eq:ABC_alg_equi_Y} to obtain the last relation.
Knowing that $\x=\B\underline{\x}$ from \eqref{eq:ABC_alg_equi_X}, we complete the proof.

\subsection{Proof of Lemma~\ref{lem:basic_inequality_II_sublinear}}
Invoking Lemma~\ref{lem:basic_inequality_I_sublinear} and using the identity 
\[
2\innprod{a-b}{a-c}=\norm{a-b}^2-\norm{b-c}^2+\norm{a-c}^2,
\]
we have that
\begin{align}
& \phi(X^{k+1},Y) \notag\\
& \leq \phi(X,Y)-\frac{1}{2\gamma}\left(\norm{X^{k+1}-X}^2_{D}-\norm{X^k-X}^2_{D} \right) \notag\\
& ~~~ -\frac{1}{\gamma}\norm{\underline{X}^{k+1}}_{B-BC-AB}^2 -\norm{X^{k+1}-X^k}_{\frac{1}{2\gamma}D - \frac{L}{2}I }^2 \notag\\
&~~~-\frac{\gamma}{2}(\norm{\underline{Y}^{k+1}-Y}^2_{B'}-\norm{\underline{Y}^k-Y}^2_{B'}+\norm{\underline{Y}^{k+1}-\underline{Y}^k}^2_{B'})\notag\\
&\overset{(a)}{=} \phi(X,Y)-\frac{1}{2\gamma}\Big(\norm{X^{k+1}-X}^2_D-\norm{X^k-X}^2_D\Big)\notag\\
& ~~~ -\frac{1}{\gamma}\norm{\underline{X}^{k+1}}_{ B- \frac{1}{2}BC-A B}^2 -\norm{X^{k+1}-X^k}_{\frac{1}{2\gamma}D - \frac{L}{2}I }^2\notag\\
&~~~-\frac{\gamma}{2}\Big(\norm{\underline{Y}^{k+1}-Y}^2_{B'}-\norm{\underline{Y}^k-Y}^2_{B'}\Big) \notag\\
&\overset{(b)}{ \leq} \phi(X,Y)-\frac{1}{2\gamma}\Big(\norm{X^{k+1}-X}^2_D-\norm{X^k-X}^2_D\Big)\notag\\
&~~~-\frac{\gamma}{2}\Big(\norm{\underline{Y}^{k+1}-Y}^2_{B'}-\norm{\underline{Y}^k-Y}^2_{B'}\Big) \label{eq:phi_k}
\end{align}
where $(a)$ is due to the fact that $\norm{\underline{Y}^{k+1}-\underline{Y}^k}^2_{B'}=\frac{1}{\gamma^2 }\norm{\underline{X}^{k+1}}^2_{BC}$ since $\underline{Y}^{k+1}-\underline{Y}^k=1/\gamma C\underline{X}^{k+1}$ and $B' C^2 = (C+bJ)^{-1}C^2B = CB$; $(b)$ comes from that $\gamma\leq \frac{\lambda_{\min}(D)}{L}$ and $B-\frac{1}{2}BC-AB = \sqrt{B}\left(I - \frac{1}{2}C - \sqrt{B}D \sqrt{B} \right)\sqrt{B}  \succeq 0.$

Then, averaging \eqref{eq:phi_k} over $k$ from $0$ to $t-1$, we have
\begin{align}\label{eq:ave_ine}
\begin{split}
& \frac{1}{t}\sum_{k=0}^{t-1}\Big(\phi(X^{k+1},Y)-\phi(X,Y)\Big)\\
&\leq-\frac{1}{2\gamma t} \left(\norm{X^t-X}^2_D-\norm{X^0-X}^2_D \right)\\
& ~~~ -\frac{\gamma}{2 t} \left(\norm{\underline{Y}^t-Y}^2_{B'}-\norm{\underline{Y}^0-Y}^2_{B'}\right)\\
&  \overset{(a)}{\leq} \frac{1}{2t}\Big(\frac{1}{\gamma}\norm{X^0-X}^2_D+ \gamma \frac{1}{ \lambda_2(C)}\norm{Y}^2_B\Big) \\
&  \overset{(b)}{=} \frac{1}{2t}\Big(\frac{1}{\gamma}\norm{X^0-X}^2_D+ \gamma \frac{1}{ \lambda_2(C)}\norm{Y}^2_{B-J}\Big) \\
&  \leq \frac{1}{2t}\Big(\frac{1}{\gamma}\norm{X^0-X}^2_D+ \gamma \frac{\rho(B-J)}{ \lambda_2(C)}\norm{Y}^2 \Big)
\end{split}
\end{align}
where we used: (a) $Y^0 = 0$ and $\lambda_{\max}\left((C+bJ)^{-1}\right) = 1/ \lambda_{\min}\left(C+bJ\right)= 1/\lambda_2(C)$ due to $C \preceq 2I$; (b) $Y\in\Span{\ones}^{\perp}$.  
Using the convexity of $\phi$ we complete the proof.

\subsection{Proof of Theorem~\ref{thm:sublinear_prox}}
Setting $B=I$, Algorithm~\eqref{alg:ABC_sublinear_prox} we study becomes
\begin{equation}\label{eq:ABC_prox}
\begin{aligned}
\underline{X}^{k+1}&= DX^k-\gamma ( \nabla f(X^k)+\underline{Y}^k), \\
X^{k+1} & = \prox{\gamma g}{\underline{X}^{k+1}},\\
\underline{Y}^{k+1}&=\underline{Y}^k+\frac{1}{\gamma} CX^{k+1}.
\end{aligned}
\end{equation}
The structure of this proof is similar to the proof of Theorem \ref{thm:unified_alg_sublinear_no_reg}.  We first establish two fundamental inequalities that are valid for any pair $(X,Y)$ such that $X\in\Span{\ones}$ and $Y\in\Span{C}$ (cf. Lemma~\ref{lem:basic_inequality_I_sublinear_prox} and Lemma~\ref{lem:basic_inequality_II_sublinear_prox}); and then apply these results with $X = X^\star$ and two choices of $Y$ to get the result of the sublinear convergence and rate separation.

\begin{lem}\label{lem:basic_inequality_I_sublinear_prox}
	Consider the setting of Theorem~\ref{thm:sublinear_prox}, let $\{ X^k, \underline{X}^k, \underline{Y}^k\}_{k \in \mathbb{N}_+}$ be the sequence generated by Algorithm~\eqref{eq:ABC_prox} under Assumption~\ref{assum:sublinear_prox}.   Then for all $ \,X\in\Span{\ones}$ and $ Y\in\Span{C}$ it holds
	\[
	\begin{aligned}
	& \phi(X^{k+1},Y) \leq \phi(X,Y) +  \frac{1}{\gamma} \innprod{X^k - X^{k+1}}{X^{k+1}-X}_D \\
	& ~~~ -  \innprod{\underline{Y}^{k+1}-Y}{X^{k+1}-X} + \frac{L}{2}\norm{X^{k+1}-X^k}^2  \\
	& ~~~ - \frac{1}{\gamma} \norm{X^{k+1}}^2_{I-C-D}.
	\end{aligned}
	\]
\end{lem}
\begin{proof}
The  proof is similar to that of Lemma~\ref{lem:basic_inequality_I_sublinear}.
\begin{align*}
& f(X^{k+1})\\
& \leq f(X)+\innprod{\nabla f(X^k)}{X^{k+1}-X}+\frac{L}{2}\norm{X^{k+1}-X^k}^2 \\
& =  f(X) + \frac{1}{\gamma} \innprod{DX^k}{X^{k+1}-X} -  \innprod{\underline{Y}^{k+1}}{X^{k+1}-X} \\
& ~~~ - \frac{1}{\gamma} \innprod{\underline{X}^{k+1}-C X^{k+1}}{X^{k+1}-X}+\frac{L}{2}\norm{X^{k+1}-X^k}^2 \\
& =  f(X) +  \frac{1}{\gamma} \innprod{D(X^k - X^{k+1})}{X^{k+1}-X} \\
& ~~~ -  \innprod{\underline{Y}^{k+1}}{X^{k+1}-X} + \frac{L}{2}\norm{X^{k+1}-X^k}^2  \\
& ~~~ - \frac{1}{\gamma} \innprod{\underline{X}^{k+1}-(C+D) X^{k+1}}{X^{k+1}-X}  .
\end{align*}
According to $X^{k+1} = \prox{\gamma g}{\underline{X}^{k+1}}$, we have $g(X^{k+1}) - g(X) \leq \frac{1}{\gamma} \innprod{ \underline{X}^{k+1} - X^{k+1}}{X^{k+1}-X}.$  We define $\phi(X,Y) = f(X) + g(X)+  \innprod{X}{Y}.$  Then we have for $X\in\Span{\ones}$ and $Y\in \Span{C}$,
\begin{align}\label{eq:prox_sublinear_I}
& \phi(X^{k+1},Y) \leq \phi(X,Y) +  \frac{1}{\gamma} \innprod{D(X^k - X^{k+1})}{X^{k+1}-X} \notag\\
& ~~~ -  \innprod{\underline{Y}^{k+1}-Y}{X^{k+1}-X} + \frac{L}{2}\norm{X^{k+1}-X^k}^2  \notag\\
& ~~~ - \frac{1}{\gamma} \innprod{(I-C-D) X^{k+1}}{X^{k+1}-X} \notag\\
& =\phi(X,Y) +  \frac{1}{\gamma} \innprod{X^k - X^{k+1}}{X^{k+1}-X}_D \notag\\
& ~~~ -  \innprod{\underline{Y}^{k+1}-Y}{X^{k+1}-X} + \frac{L}{2}\norm{X^{k+1}-X^k}^2  \notag\\
& ~~~ - \frac{1}{\gamma} \norm{X^{k+1}}^2_{I-C-D}
\end{align}
\end{proof}

\begin{lem}\label{lem:basic_inequality_II_sublinear_prox}
	Under the same conditions as Lemma~\ref{lem:basic_inequality_I_sublinear_prox}, if $\gamma\leq \frac{\lambda_{\min}(D)}{L}$, then for all $X\in\Span{\ones}$ and $Y\in\Span{C}$ it holds
	\begin{equation}\label{eq:ineq_prox_i}
	\begin{aligned}
	& \phi(\widehat{X}^{t},Y)-\phi(X,Y)\\
	& \leq \frac{1}{2t}\Big(\frac{1}{\gamma}\norm{X^0-X}^2_D+\gamma\frac{1}{\lambda_2(C)}\norm{Y}^2\Big)  .
	\end{aligned}
	\end{equation}
\end{lem}
\begin{proof}
Continuing from \eqref{eq:prox_sublinear_I}, we have
\begin{align*}
& \phi(X^{k+1},Y) \\
& \leq \phi(X,Y)  - \frac{1}{2\gamma} \left( \norm{X^{k+1}-X}_D^2 -  \norm{X^{k}-X}_D^2\right) \\ 
& ~~~ -  \innprod{\underline{Y}^{k+1}-Y}{CX^{k+1}}_{(J+C)^{-1}} - \norm{X^{k+1}-X^k}^2_{\frac{1}{2\gamma}D-\frac{L}{2}} \\
& ~~~ - \frac{1}{\gamma} \norm{X^{k+1}}_{I-C-D}\\
& = \phi(X,Y)  - \frac{1}{2\gamma} \left( \norm{X^{k+1}-X}_D^2 -  \norm{X^{k}-X}_D^2\right) \\
& ~~~ -\frac{\gamma}{2} \left(\norm{\underline{Y}^{k+1}-Y}^2_{(J+C)^{-1}} - \norm{\underline{Y}^{k}-Y}^2_{(J+C)^{-1}} \right) \\
& ~~~  - \frac{1}{\gamma} \norm{X^{k+1}}_{I-\frac{C}{2}-D}  - \norm{X^{k+1}-X^k}^2_{\frac{1}{2\gamma}D-\frac{L}{2}} \\
& \leq \phi(X,Y)  - \frac{1}{2\gamma} \left( \norm{X^{k+1}-X}_D^2 -  \norm{X^{k}-X}_D^2\right) \\
& ~~~ -\frac{\gamma}{2} \left(\norm{\underline{Y}^{k+1}-Y}^2_{(J+C)^{-1}} - \norm{\underline{Y}^{k}-Y}^2_{(J+C)^{-1}} \right),
\end{align*}
where the last step is due to that $I-\frac{C}{2}-D \succeq 0$ and $\gamma\leq \frac{\lambda_{\min}(D)}{L}.$  
Then, averaging the above over $k$ from $0$ to $t-1$, we have
\[
\begin{aligned}
& \frac{1}{t}\sum_{k=0}^{t-1}\Big(\phi(X^{k+1},Y)-\phi(X,Y)\Big)\\
&\leq-\frac{1}{2\gamma t} \left(\norm{X^t-X}^2_D-\norm{X^0-X}^2_D \right)\\
& ~~~ -\frac{\gamma}{2 t} \left(\norm{\underline{Y}^t-Y}^2_{(J+C)^{-1}}-\norm{\underline{Y}^0-Y}^2_{(J+C)^{-1}}\right)\\
& \leq \frac{1}{2t}\Big(\frac{1}{\gamma}\norm{X^0-X}^2_D+ \gamma \frac{1}{ \lambda_2(C)}\norm{Y}^2\Big) .
\end{aligned}
\]
Using the convexity of $\phi$  completes the proof.
\end{proof}

For notational simplicity, we set $r(X) = f(X)+g(X).$
From \eqref{eq:ineq_prox_i}, we have
\begin{equation}\label{eq:ineq_prox_ii}
\begin{aligned}
& \phi(\widehat{X}^{t},Y)-\phi(X^\star,Y)  = r(\widehat{X}^{t})  - r(X^\star)- \innprod{\widehat{X}^{t} - X^\star}{Y}\\ 
& =   r(\widehat{X}^{t})   - r(X^\star)- \innprod{\widehat{X}^{t} }{Y}  \leq h(\norm{Y}),
\end{aligned}
\end{equation}
where $h(\cdot)=\frac{1}{2t}\Big(\frac{1}{\gamma}\norm{X^0-X^\star}^2_D+\gamma\frac{1}{\lambda_2(C)}(\cdot)^2\Big)$.
Now setting $Y=-2\frac{(I-J)\widehat{X}^t}{\norm{(I-J)\widehat{X}^t}}\norm{Y^\star }$.  The rest of the proof is similar to that in Theorem \ref{thm:unified_alg_sublinear_no_reg}.



\bibliographystyle{IEEEtran}
\bibliography{IEEEabrv,reference}

\end{document}